\documentclass[12pt]{amsart}
\usepackage{amsfonts,amsmath,amssymb,epsfig,cite,graphicx,hyperref,
color, esint, fancyhdr, enumerate, latexsym,amsrefs}

\makeatletter
\@namedef{subjclassname@2020}{%
  \textup{2020} Mathematics Subject Classification}
\makeatother

\usepackage[T1]{fontenc}

\newtheorem{thm}{Theorem}[section]
\newtheorem{prop}[thm]{Proposition}
\newtheorem{lem}[thm]{Lemma}
\newtheorem{cor}[thm]{Corollary}
\newtheorem{claim}{\textbf{Claim}}
\theoremstyle{definition}

\newcommand{\ep}{\epsilon}
\newcommand{\mbb}{\mathbb}
\newcommand{\de}{\delta}
\newcommand{\ov}{\overline}
\newcommand{\La}{\Lambda}
\newcommand{\pa}{\partial}
\newcommand{\mf}{\mathbb}

\newcommand{\Om}{\Omega}
\newcommand{\al}{\alpha}
\newcommand{\be}{\beta}
\newcommand{\ga}{\gamma}
\newcommand{\z}{\zeta}
\newcommand{\la}{\lambda}
\newcommand{\ti}{\tilde}
\newcommand{\De}{\Delta}
\newcommand{\Ga}{\Gamma}

\renewcommand{\Re}{\operatorname{Re}}
\renewcommand{\Im}{\operatorname{Im}}
\newcommand{\supp}{\operatorname{supp}}
\newcommand{\Ric}{\operatorname{Ric}}
\newcommand{\diag}{\operatorname{diag}}

\newcommand{\adj}{\operatorname{adj}}
\newcommand{\tr}{\operatorname{trace}}

\numberwithin{equation}{section}

\hypersetup{
    colorlinks,
    citecolor=blue,
    filecolor=black, 
    linkcolor=red,
    urlcolor=black
}

\dedicatory{In memory of M. S. Narasimhan and R. R. Simha}

\begin{document}

\baselineskip=17pt

\title[Narasimhan-Simha type metrics]{Narasimhan--Simha type metrics on strongly pseudoconvex domains in $\mathbb{C}^n$}

\author[D. Borah]{Diganta Borah}
\address{Indian Institute of Science Education and Research, Pune, India}
\email{dborah@iiserpune.ac.in}

\author[K. Verma]{Kaushal Verma}
\address{Department of Mathematics, Indian Institute of Science, Bangalore, India}
\email{kverma@iisc.ac.in}

\date{}

\begin{abstract}
For a bounded domain $D \subset \mathbb{C}^n$, let $K_D = K_D(z) > 0$ denote the Bergman kernel on the diagonal and consider the reproducing kernel Hilbert space of holomorphic functions on $D$ that are square integrable with respect to the weight $K_D^{-d}$, where $d \geq 0$ is an integer. The corresponding weighted kernel $K_{D, d}$ transforms appropriately under biholomorphisms and hence produces an invariant K\"{a}hler metric on $D$. Thus, there is a hierarchy of such metrics starting with the classical Bergman metric that corresponds to the case $d=0$. This note is an attempt to study this class of metrics in much the same way as the Bergman metric has been with a view towards identifying properties that are common to this family. When $D$ is strongly pseudoconvex, the scaling principle is used to obtain the boundary asymptotics of these metrics and several invariants associated to them. It turns out that all these metrics are complete on strongly pseudoconvex domains.
\end{abstract}



\maketitle

\section{Introduction}
For a bounded domain $D\subset\mf{C}^n$ and a non-negative measurable function $\phi$
on it, consider the space $L^2_{\phi}(D)$ of all measurable functions on $D$ that are square integrable with respect to the measure $\phi \,dV$ where $dV$ is the Lebesgue measure on $\mf{C}^n$. It is known that if $1/\phi \in L^{\infty}_{\text{loc}}(D)$ then the weighted Bergman space
\[
A^2_{\phi}(D)=\left\{\text{$f: D \to \mf{C}$ holomorphic and $\int_{D} \vert f\vert^2 \phi \, dV< \infty$} \right\}
\]
is a closed subspace of $L^2_{\phi}(D)$---see \cite{PWZ} and therefore there is a reproducing kernel denoted by $K_{D, \phi}(z,w)$, that makes $A^2_{\phi}(D)$ a reproducing kernel Hilbert space. The kernel $K_{D,\phi}(z,w)$ is uniquely determined by the following properties: $K_{D,\phi}(\cdot, w) \in A^2_{\phi}(D)$ for each $w \in D$, it is anti-symmetric, i.e., $K_{D,\phi}(z,w)=\ov{K_{D,\phi}(w,z)}$, and it reproduces $A^2_{\phi}(D)$:
\[
f(z)=\int_{D} K_{D,\phi}(z,w) f(w) \phi(w) \, dV(w), \quad f \in A^2_{\phi}(D).
\]
It also follows that for any complete orthonormal basis $\{\phi_k\}$ of $A^2_{\phi}(D)$,
\[
K_{D,\phi}(z,w)=\sum_{k} \phi_k(z) \ov{\phi_{k}(w)},
\]
where the series converges uniformly on compact subsets of $D \times D$.

\medskip

Let $K_D(z,w)$ be the usual Bergman kernel for $D$, and $K_D(z)=K_D(z,z)$ its restriction to the diagonal. Since $D$ is bounded, $K_D(z)>0$ everywhere and hence, for an integer $d\geq 0$, we may consider $\phi=K_D^{-d}$ as a weight to which all the considerations mentioned above can be clearly applied. In this case, the weighted Bergman space will be denoted by $A^2_d(D)$, i.e.,
\[
A^2_d(D)=\Big\{\text{$f: D \to \mf{C}$ holomorphic and $\int_{D} \vert f\vert^2 K_D^{-d} \, dV < \infty$}\Big\},
\]
and the corresponding weighted Bergman kernel by $K_{D,d}(z,w)$ to emphasize the dependence on $d\geq0$. The fact $K_D(z) \geq 1/\text{Vol}(D)$ everywhere for a bounded domain $D$ implies that $K_{D,d}(z)=K_{D,d}(z,z)>0$ and it can also be established that $\log K_{D,d}$ is strongly plurisubharmonic in the same way as for the usual Bergman kernel (see \cite{Chen}). The transformation rule for $K_{D,d}(z,w)$ under a biholomorphism $F: D \to D'$ is
\begin{equation}\label{tr-1}
K_{D,d}(z,w) =\big(\det F^{\prime}(z)\big)^{d+1} K_{D', d}\big(F(z),F(w)\big) \big(\ov{\det F^{\prime}(w)}\big)^{d+1}
\end{equation}
as will be seen later. Here, $F'(z)$ is the complex Jacobian matrix of $F$ at $z$. Therefore, $K_{D,d}(z)$ induces an invariant K\"ahler metric
\begin{equation}\label{ns-metric}
ds_{D,d}^2=\sum_{\al,\be=1}^n g_{\al\ov \be}^{D,d}(z) \,dz_{\al}\,d\ov z_{\be},
\end{equation}
where
\begin{equation}\label{ns-comp}
g^{D,d}_{\al\ov \be}(z)=\frac{\pa^2}{\pa z_{\al} \pa \ov z_{\be}} \log K_{D,d}(z).
\end{equation}
Several aspects of weighted Bergman spaces have been considered even with weight $K_D^{-d}$. However, the observation that the choice of $K_D^{-d}$, $d \geq 0$, as a weight gives rise to a kernel that transforms correctly under a change of coordinates, and hence induces an invariant metric, goes back to the work of Narashiman-Simha \cite{NS} on the moduli space of compact complex manifolds with ample canonical bundle. Adapting a special case of their construction on a bounded domain in $\mbb{C}^n$ leads to the consideration of $A^2_d(D)$ and hence $ds^2_{D,d}$ as in \eqref{ns-metric}. Henceforth, we will refer to $K_{D, d}$ as a weighted kernel of order $d$, and $ds^2_{D,d}$ as a Narasimhan-Simha type metric of order $d$. It must be mentioned that some general properties of this metric on pseudoconvex domains in complex manifolds were studied by Chen \cite{Chen} and more recent applications of this metric can be found in \cite{Ber-Pau}, \cite{Pau-Tak}, and \cite{To-Yeu} for example.

\medskip

The purpose of this note is to study some aspects of these metrics on $C^2$-smoothly bounded strongly pseudoconvex domains. In what follows, the integer $d\geq 0$ will be fixed once and for all. Let $\tau_{D,d}(z,v)$ be the length of a vector $v \in \mbb{C}^n$ at $z \in D$ in the metric $ds^2_{D,d}$. The Riemannian volume element of this metric will be denoted by
\[
g_{D,d}(z)=\det G_{D,d}(z) \quad \text{where} \quad G_{D,d}(z)= \begin{pmatrix}g^{D,d}_{\al \ov \be}(z)\end{pmatrix},
\]
and note that
\[
\beta_{D,d}(z)=g_{D,d}(z)\big(K_{D,d}(z)\big)^{-1/(d+1)}
\]
is a biholomorphic invariant. The holomorphic sectional curvature of $ds^2_{D,d}$ is given by
\[
R_{D,d}(z,v)=\frac{\sum_{\al,\be,\ga,\de} R^{D,d}_{\ov \al \be \ga \ov \de}(z)\ov v^{\al} v^{\be} v^{\ga} \ov v^{\de}}{\big(\sum_{\al,\be} g^{D,d}_{\al\ov \be}(z)v^\al \ov v^\be\big)^2},
\]
where
\[
R^{D,d}_{\ov \al \be \ga \ov \de}=-\frac{\pa^2 g^{D,d}_{\be \ov \al}}{\pa z_{\ga} \pa \ov z_{\de}}+\sum_{\mu,\nu} g_{D,d}^{\nu \ov \mu}\frac{\pa g^{D,d}_{\be \ov \mu}}{\pa z_{\ga}}\frac{\pa g^{D,d}_{\nu \ov \al}}{\pa \ov z_{\de}},
\]
$g_{D,d}^{\nu \ov \mu}(z)$ being the $(\nu,\mu)$th entry of the inverse of the matrix $G_{D,d}(z)$. The Ricci curvature of $ds^2_{D,d}$ is given by
\[
\text{Ric}_{D,d}(z,v)=\frac{\sum_{\al, \be} \Ric_{\al \ov \be}^{D,d}(z) v^{\al} \ov v^{\be}}{\sum_{\al,\be} g^{D,d}_{\al\ov \be}(z)v^\al \ov v^\be},
\]
where
\[
\Ric_{\al \ov \be}^{D,d}=  - \frac{\pa^2}{\pa z_{\al} \pa \ov z_{\be}}\log g_{D,d}.
\]
Finally, for $p \in D$, let $\de_D(p)$ be the Euclidean distance from $p$ to the boundary $\pa D$. For $p$ close to the boundary $\pa D$, let $\ti p \in \pa D$ be such that $\de_D(p) = \vert p - \ti p \vert$, and for a tangent vector $v \in \mf C^n$ based at $p$, write $v = v_H(p) + v_N(p)$ where this decomposition is taken along the tangential and normal directions respectively at $\ti p$. We will also use the standard notation
\[
z=({'z},z_n), \quad {'z}=(z_1, \ldots, z_{n-1}).
\]

\begin{thm}\label{spscvx}
Let $D \subset \mf{C}^n$ be a $C^2$-smoothly bounded strongly pseudoconvex domain and let $p^0 \in \pa D$. Denote by
\[
D_{\infty}=\big\{z \in \mf{C}^n: 2\Re z_n+\vert {'}z\vert^2<0 \big\},
\]
the unbounded realization of the unit ball in $\mf{C}^n$ and $b^{*}=('0,-1) \in D_{\infty}$. Then there are local holomorphic coordinates near $p^0$ in which we have
\begin{align*}
&\text{\em (a)}& &\de_D(p)^{(d+1)(n+1)} K_{D,d}(p) \to K_{D_{\infty},d}(b^{*})= c\left(\frac{n!}{2^{n+1}\pi^n}\right)^{d+1}, &\\
&\text{\em (b)}& &\de_D(p)^{n+1} g_{D,d}(p) \to g_{D_{\infty},d}(b^*)=  \frac{1}{2^{n+1}} (d+1)^n (n+1)^n,&\\
&\text{\em (c)}& &\be_{D,d}(p) \to \be_{D_{\infty},d}(b^*)=(d+1)^{n}(n+1)^{n} \left(\frac{1}{c}\right)^{1/(d+1)} \frac{\pi^n}{n!},&\\
&\text{\em (d)}& &\de_D(p) \tau_{D,d}(p,v) \to \tau_{D_{\infty},d}\big(b^{*}, ({'0},v_n)\big)= \frac{1}{2}\sqrt{(d+1)(n+1)}\big\vert v_N(p^0)\big\vert, &\\
&\text{\em (e)}& &\sqrt{\de_D(p)} \tau_{D,d}\big(p,v_H(p)\big) \to \tau_{D_\infty,d}\big(b^*,({'}v,0)\big)&\\
& & &\text{\hspace{4cm}}=\sqrt{\frac{1}{2}(d+1)(n+1)\mathcal{L}_{\pa D}\big(p^0, v_H(p^0)\big)}, &\\
&\text{\em (f)}& &R_{D,d}(p,v)\to -\frac{2}{(d+1)(n+1)}, \quad v \in \mf{C}^n \setminus\{0\}, &
\\
&\text{\em (g)}& &\Ric_{D,d}(p,v)\to -\frac{1}{d+1}, \quad v \in \mf{C}^n \setminus\{0\}. &
\end{align*}
Here
\[
c=\frac{\Ga\big((d+1)(n+1)\big)}{n!\Ga\big(d(n+1)+1\big)},
\]
and $\mathcal{L}_{\pa D}$ denotes the Levi form of $\pa D$ with respect to some defining function for $D$.
\end{thm}

Several remarks are in order. First, the Cayley transform
\begin{equation}\label{cayley}
\Phi(z)= \left(\frac{\sqrt{2}\,{'z}}{z_n-1}, \frac{z_n+1}{z_n-1}\right)
\end{equation}
maps $D_{\infty}$ biholomorphically onto the unit ball $\mf B^n \subset \mf{C}^n$ with $\Phi(b^*)=0$. A computation, which is given later, shows that $K_{\mbb{B}^n,d}=cK_{\mbb{B}^n}^{d+1}$ with $c$ as defined in the theorem above, and this is what enables the explicit limiting expressions above to be computed. Thus Theorem~\ref{spscvx} (a) shows that
\[
K_{D,d}(z) \sim \big(K_{D}(z)\big)^{d+1}
\]
where $z$ is close to the boundary of a $C^2$-smooth strongly pseudoconvex domain. Thus for a fixed $d \geq 0$, the weighted Bergman kernel on the diagonal $K_{D,d}(z)$ behaves like the $(d+1)$-st power of the usual Bergman kernel $K_D(z)$ as $z$ varies near $\pa D$.

\medskip

The question of determining the behaviour of $K_{D,d}(z)$ at a given fixed $z \in D$ is known and can be read off from several existing theorems \cites{Eng97,Eng02, Chen06} of which Theorem~1.1 in Chen \cite{Chen06} is a convenient prototype. Indeed, and application of this theorem (set $\phi=K_{D}^{-1}$ and $\psi=1$ in the notation therein) shows that
\[
\lim_{d \to \infty} \left(K_{D,d}(z)\right)^{1/d} = K_{D}(z)
\]
for each fixed $z \in D$ and this translates to $K_{D,d}(z) \sim \big(K_{D}(z)\big)^d$ for $d$ large.

\medskip

To summarize, if $K_{D,d}(z)$ is regarded as a function of $z$ and $d$, its behaviour as a function of $z$ is like $\big(K_{D}(z)\big)^{d+1}$ and this is furnished by Theorem~\ref{spscvx}~(a), whereas its behaviour as a function of $d$ is like $\big(K_D(z)\big)^d$ and this is already known. A very recent result that is relevant to Theorem~\ref{spscvx}~(a) can be found in Chen \cite{Chen20}.

\medskip

Second, Theorem~\ref{spscvx}~(d), (e) are direct analogs of Graham's result \cite{Gr} for the Kobayashi and Carath\'{e}odory metrics. Further, Theorem~\ref{spscvx}~(f) shows that the holomorphic sectional curvature of $ds^2_{D,d}$ along a sequence $(p_j,v_j) \in D \times \mf{C}^n$ approaches $-2/((d+1)(n+1))$ as 
$(p_j,v_j) \to (p,v) \in \pa D \times \mf{C}^n$. A verbatim application of Theorem~1.17 in \cite{GK} shows that if $D,D' \subset \mbb{C}^n$ are $C^2$-smooth strongly pseudoconvex domains equipped with the metrics $ds^2_{D,d}$ and $ds^2_{D',d}$, then every isometry $f: (D, ds^2_{D,d}) \to (D', ds^2_{D',d})$ is either holomorphic or conjugate holomorphic.

\medskip

\begin{thm}
Let $D \subset \mbb C^n$ be a $C^2$-smoothly bounded strongly pseudoconvex domain. Then, $ds^2_{D, d}$ is complete for every $d \ge 0$.
\end{thm}

%%%%%%%%%%%%%%%%%%%%%%%
\section{Some Examples}

\noindent In this section, we prove the transformation rule \eqref{tr-1} and compute the weighted kernel of order $d \ge 0$ for the unit ball and polydisc in $\mathbb{C}^n$ by working with a complete orthonormal basis, and recall Selberg's observation about this kernel on homogeneous domains.

\begin{prop}\label{tthm-1}
Let $F: D \to D'$ be a biholomorphic map. Then the transformation rule \eqref{tr-1} holds.
\end{prop}

\begin{proof}
It is enough to show that the right hand side of \eqref{tr-1} reproduces $A^2_d(D)$, i.e., if $f \in A^2_d(D)$ then
\begin{equation}\label{rep-ns}
f(z)=\int_D \big(\det F^{\prime}(z)\big)^{d+1} K_{D',d}\big(F(z),F(w)\big) \ov{\big(\det F^{\prime}(w)\big)^{d+1}} f(w) K_D^{-d}(w) \, dV(w)
\end{equation}
for $z \in D$. By the transformation rule for the Bergman kernel, we have
\[
\ov{\big(\det F^{\prime}(w)\big)^{d+1}} K_D^{-d}(w)
=\frac{1}{(\det F'(w))^{d+1}} K_{D'}^{-d}\big(F(w)\big) \big\vert \det F'(w) \big\vert^{2}.
\]
Therefore, substituting $\ti w =F(w)$, and using $dV(\ti w) = \big\vert \det F'(w) \big\vert^{2} dV(w)$, the right hand side of \eqref{rep-ns} becomes
\[
\big(\det F^{\prime}(z)\big)^{d+1} \int_{D'} K_{D',d}\big(F(z), \ti w) \frac{f\circ F^{-1}(\ti w)}{\det  F^{\prime}\big(F^{-1}(\ti w)\big)^{d+1}} K_{D'}^{-d}(\ti w) \, dV(\ti w),
\]
and by the reproducing property of $K_{D',d}$, this is equal to
\[
\big(\det F^{\prime}(z)\big)^{d+1} \frac{f\circ F^{-1}\big(F(z)\big)}{\big(\det F^{\prime}(z)\big)^{d+1}}= f(z),
\]
which is the left hand side of \eqref{rep-ns}.
\end{proof}

The transformation rule implies that the Narasimhan--Simha metric is invariant under biholomorphisms. Indeed, if $F: D \to D'$ is a biholomorphism, then from \eqref{tr-1},
\[
\log K_{D,d}(z)= (d+1) \log \big\vert \det F^{\prime}(z)\big\vert^2 + \log K_{D',d}\big(F(z)\big),
\]
and since the first term on the right is locally the sum of a holomorphic and antiholomorphic function, its mixed partial derivatives vanish, and hence
\[
g_{\al\ov \be}^{D,d}(z)=\sum_{l,m=1}^n g_{l \ov m}^{D',d}\big(F(z)\big) \frac{\pa F_l}{\pa z_{\al}}\ov {\left( \frac{\pa F_m}{\pa z_{\be}} \right)}.
\]
In terms of matrices, this takes the form
\begin{equation}\label{tr-2}
G_{D,d}(z)=F^{\prime}(z)^TG_{D',d}\big(F(z)\big) \ov{F^{\prime}(z)},
\end{equation}
and it follows that for $v \in \mf{C}^n$,
\begin{equation}\label{tr-3}
\tau_{D,d}(p, v)=\tau_{D', d}\big(F(p), F'(p)v\big),
\end{equation}
establishing the invariance of this metric. Here, $F^{\prime}(z)^T$ is the transpose of the matrix $F'(z)$. We also emphasise that throughout this article, vectors in $\mf{C}^n$ will be regarded as column vectors.

\medskip

To compute this weighted kernel on the unit ball $\mf{B}^n$ and the unit polydisc ${\De}^n$ in $\mf{C}^n$, we will need the following:
\begin{lem}\label{Dir-int}
Let $f$ be a continuous function on $[0,1]$. Then for $r \in \mf{R}^n$ with $r_i \geq 0$, and multiindex $\al$,
\begin{align*}
& \text{\em (a)} \quad \int_{r_1+\ldots +r_n<1} f(r_1+\cdots +r_n) r^{\al-1} \,dr= \frac{\prod\Ga(\al_i)}{\Ga(\vert \al\vert)} \int_0^1 f(\tau) \tau^{\vert \al\vert-1} \,d\tau.
\\
& \text{\em (b)} \quad \int_{r_1^2+\ldots+r_n^2<1} f(r_1^2+\cdots+r_n^2)r^{2\al-1} \,dr=\frac{1}{2^n} \frac{\prod \Ga(\al_i)}{\Ga(\vert \al\vert)} \int_0^1 f(\tau) \tau^{\vert \al\vert-1} \,d\tau.
\end{align*}
\end{lem}
For (a) see Section 12.5 of \cite{WW} and (b) follows by substituting $r_i^2=t_i$ in (a).

\begin{prop}\label{NS-ball}
For the unit ball $\mf{B}^n \subset \mf{C}^n$ and $d \ge 0$
\[
K_{\mf{B}^n,d} = \frac{\Ga\big((d+1)(n+1)\big)}{n!\Ga(d(n+1)+1)} K_{\mf{B}^n}^{d+1}.
\]
\end{prop}

\begin{proof}
Recall that
\[
K_{\mf{B}^n}(z,w)=\frac{n!}{\pi^n}\frac{1}{\big(1-\langle z,w\rangle\big)^{n+1}}.
\]
A basis for the space $A^2_d(D)$ is $\{z^{\al}\}$ and by integrating in polar coordinates we see that it is orthogonal. Indeed, 
\begin{align*}
& \frac{\pi^{dn}}{(n!)^d}\int_{\mf{B}^n} z^{\al} \ov z^{\be} \big(1-\vert z \vert^2\big)^{d(n+1)} \, dV(z)\\
= &\int_{\vert r\vert^2<1}\int_{[0,2\pi]^n} r^{\al+\be+1} \big(1-\vert r\vert^2\big)^{d(n+1)}e^{i(\al-\be)\theta}\, dr  d\theta,
\end{align*}
and the integral with respect to $d\theta_i$ vanishes unless $\al_i=\be_i$. Also, if $\ga_{\al}=\|z^{\al}\|_{A^2_d(\mf{B}^n)}^2$, we obtain in view of Lemma \ref{Dir-int} that
\begin{align*}
\ga_{\al} & = \frac{\pi^{dn}}{(n!)^d} \, (2\pi)^n \int_{\vert r \vert^2<1} r^{2\al+1} \big(1-\vert r\vert^2\big)^{d(n+1)}dr\\
&  =  \frac{\pi^{(d+1)n}}{(n!)^d} \frac{\Ga\big(d(n+1)+1\big) \prod \Ga(\al_i+1)}{\Ga\big((d+1)(n+1)+\vert \al \vert\big)}.
\end{align*}
Using the series expansion
\[
\frac{1}{\big(1-\langle z, w\rangle\big)^s} = \sum_{\al} \frac{\Ga\big(s+\vert \al\vert\big)}{\Ga(s) \prod \Ga(\al_i+1)} (z\ov w)^{\al},
\]
we now obtain
\begin{align*}
K_{\mf{B}^n,d}(z,w) & = \sum_{\al} \frac{1}{\ga_{\al}} z^{\al} \ov w^{\al} = \frac{(n!)^d}{\pi^{(d+1)n}} \sum_{\al} \frac{\Ga\big((d+1)(n+1)+\vert \al \vert\big)}{\Ga\big(d(n+1)+1\big)\prod \Ga(\al_i+1)} (z \ov w)^{\al}\\
& =  \frac{(n!)^d}{\pi^{(d+1)n}} \frac{\Ga\big((d+1)(n+1)\big)}{\Ga\big(d(n+1)+1\big)} \frac{1}{\big(1-\langle z, w\rangle\big)^{(d+1)(n+1)}}\\ & = \frac{\Ga\big((d+1)(n+1)\big)}{n!\Ga\big(d(n+1)+1\big)} K_{\mf{B}^n}^{d+1}(z, w)
\end{align*}
as required.
\end{proof}

\begin{cor}\label{NS-ball-comp-z}
For the unit ball $\mf{B}^n \subset \mf{C}^n$, we have
\begin{align*}
& K_{\mf{B}^n,d}(z) = c\left(\frac{n!}{\pi^n}\right)^{d+1}\frac{1}{(1-\vert z \vert^2)^{(d+1)(n+1)}},\\
& g_{\al \ov \be}^{\mf{B}^n,d}(z)=(d+1)(n+1)\left(\frac{\de_{\al \be}}{1-\vert z \vert^2}+\frac{\ov z_{\al} z_{\be}}{(1-\vert z \vert^2)^2}\right),\\
& g_{\mf{B}^n,d}(z) =\frac{(d+1)^n(n+1)^n}{(1-\vert z\vert^2)^{n+1}},\\
& \be_{\mf{B}^n,d}(z) =(d+1)^{n}(n+1)^{n} \left(\frac{1}{c}\right)^{1/(d+1)} \frac{\pi^n}{n!},\\
& R_{\mf{B}^n,d}(z,v) =-\frac{2}{(d+1)(n+1)},\\
& \Ric_{\mf{B}^n,d}(z,v) =-\frac{1}{d+1},
\end{align*}
for $z \in \mf{B}^n$ and $v \in \mf{C}^n \setminus\{0\}$, and where $c$ is the constant in the statement of Theorem~\ref{spscvx}.
\end{cor}
\begin{proof}
The expression for $K_{\mf{B}^n,d}(z)$ follows from the previous proposition. From this expression,
\begin{equation}\label{g_al-be-as-der}
g_{\al \ov \be}^{\mf{B}^n,d}(z)= (d+1)(n+1)\frac{\pa^2}{\pa z_{\al} \pa \ov z_{\be}} \log\frac{1}{1-\vert z\vert^2},
\end{equation}
and the required expression follows upon differentiation. Now observe that
\[
g_{D,d}(z)=\frac{(d+1)^n(n+1)^n}{(1-\vert z \vert^2)^{2n}} \det \Big((1-\vert z \vert^2)\mf{I} +A_z\Big),
\]
where $\mf{I}=\mf{I}_{n}$ is the identity matrix of size $n$ and $A_z$ is the matrix $A_z=(\ov z_{\al}z_{\be})=\ov z z^T$. Since the rank of $A_z$ is $\leq 1$, at least $n-1$ eigenvalues of $A_z$ are $0$, and hence the remaining eignevalue is $\tr A=\vert z \vert^2$. Therefore, its characteristic polynomial is $P_{A_z}(\la)=\det(\la \mf{I}-A_z)=\la^{n}-\vert z\vert^2 \la^{n-1}$. From this, we obtain
\[
\det\big((1-\vert z \vert^2)\mf{I}+A_z\big)=(-1)^n P_{A_z}(-(1-\vert z\vert^2))=(1-\vert z\vert^2)^{n-1},
\]
and hence the desired expression for $g_{\mf{B}^n,d}(z)$ follows. The expression for $\be_{\mf{B}^n,d}(z)$ is now immediate from its definition. Since the automorphism group of $\mf{B}^n$ acts transitively on $\mf{B}^n$ and contains the unitary rotations, the holomorphic sectional curvature of $ds^2_{\mf{B}^n,d}$ is constant and hence it is enough to compute
\[
R_{\mf{B}^n,d}\big(0,({'}0,1)\big)=\frac{1}{g^{\mf{B}^n,d}_{n \ov n}(0)^2}R^{\mf{B}^n,d}_{\ov n n n \ov n}(0).
\]
Observe that all the first order partial derivatives of $g^{\mf{B}^n,d}_{\al \ov \be}$ vanish at $0$ and so
\[
R_{\ov n n n \ov n}(0)=-\frac{\pa g^{\mf{B}^n,d}_{n\ov n}}{\pa z_n \pa \ov z_n}(0)= -2(d+1)(n+1).
\]
Also, $g_{n \ov n}(0)=(d+1)(n+1)$. Therefore,
\[
R_{\mf{B}^n,d}(z,v)=R_{\mf{B}^n,d}\big(0,({'}0,1)\big)=-\frac{2}{(d+1)(n+1)}.
\]
Finally, from the expression of $g_{\mf{B}^n,d}(z)$, we obtain
\begin{equation}\label{Ric_al-be-as-der}
\Ric^{\mf{B}^n,d}_{\al \ov \be}(z)=-(n+1)\frac{\pa^2}{\pa z_{\al}\pa \ov z_{\be}} \log\frac{1}{1-\vert z \vert^2}.
\end{equation}
From \eqref{g_al-be-as-der} and \eqref{Ric_al-be-as-der},
\[
\Ric^{\mf{B}^n,d}_{\al \ov \be}(z)=-\frac{1}{d+1}g^{\mf{B}^n,d}_{\al \ov \be}(z),
\]
which implies that $\Ric_{\mf{B}^n,d}(z,v)$ is the constant $-1/(d+1)$.
\end{proof}

For the polydisc $\De^n \subset \mbb C^n$, observe that if $D_1\subset \mf{C}^n$ and $D_2\subset \mf{C}^m$ are bounded domains, then for $(z_1,z_2), (w_1,w_2) \in D_1 \times D_2$
\begin{equation}\label{K-on-prod}
K_{D_1\times D_2,d}\big((z_1,z_2), (w_1,w_2)\big)=K_{D_1,d}(z_1,w_1) K_{D_2,d}(z_2,w_2)
\end{equation}
as the reproducing property shows. By combining this with Proposition~\ref{NS-ball}, we obtain
\begin{equation}\label{NS-polydisc}
K_{\De^n, d}=(2d+1)^n K_{\De^n}^{d+1}.
\end{equation}

\noindent In general, for a bounded domain $\Om \subset \mf C^n$ with a transitive holomorphic automorphism group, it is known that (see Selberg \cite{Sel} and note that our $K_{\Om,d}$ coincides with $K_{d+1}$ of this paper)
\[
K_{\Om, d}(z, w) = c(d+1) (K_{\Om}(z, w))^{d+1},
\]
where
\[
\frac{1}{c(d)} = \int_{\Om} \frac{\vert K_{\Om}(z, w) \vert^{2d}}{\left( K_{\Om}(z, z) K_{\Om}(w, w) \right)^d} \; K_{\Om}(z, z) \; dV.
\] 

%%%%%%%%%%%%%%%%%

\section{Localization}\label{localisation}

The localization of the usual Bergman kernel and metric are well known---see for example \cite{DFH84} for a qualitative estimate on bounded domains of holomorphy and \cite{KY} for a quantitative estimate near holomorphic peak points of such domains. We will now show that the analogs of these results hold for $K_{D,d}$ and $\tau_{D,d}$ as well. The first step in doing so is to understand their relation with certain minimum integrals.

%%%%%%%
\subsection{Minimum Integrals}
Let $D \subset \mf C^n$ be a bounded domain. For $p \in D$ and a nonzero vector $v \in \mf{C}^n$, consider the following minimum integrals:
\begin{equation}\label{min-int}
\begin{aligned}
I_{D,d}^0(p,v) & =\inf \left\{\|f\|_{A^2_d(D)}^2 : f \in A^2_d(D), f(p)=1 \right\},\\
I_{D,d}^1(p,v) & = \inf \left\{\|f\|_{A^2_d(D)}^2 : f \in A^2_d(D), f(p)=0, f^{\prime}(p)v=1 \right\},\\
I_{D,d}^2(p,v) & = \inf \left\{\|f\|_{A^2_d(D)}^2 : f \in A^2_d(D), f(p)=0, f^{\prime}(p)v=0,\right.\\
&\text{\hspace{7cm}}\left. v^{T}f''(p)v=1\right\},\\
\la_{D,d}^k(p,v) & =\inf\left\{\|f\|_{A^2_d(D)}^2 : f \in A^2_d(D), f(p)=0,\right.\\
& \text{\hspace{0cm}} \left. \frac{\pa f}{\pa z_j}(p)=0 \text{ for $1 \leq j <k$ }, \frac{\pa f}{\pa z_k}(p)=1 \right\}, 1 \leq k \leq n,\\
I_{D,d}(p,v) & = \inf \left\{\|f\|_{A^2_d(D)}^2 : f \in A^2_d(D), f(p)=0, f'(p)=0,\right.\\
&\text{\hspace{3.5cm}}\left. v^{t} f''(p) \ov G_{D,d}^{-1} (p) \ov{f''}(p)\ov v=1\right\}, \text{ and}\\
M_{D,d}(p,v) & = \inf \left\{\|f\|_{A^2_d(D)}^2 : f \in A^2_d(D), f(p)=0, f'(p)=0,\right.\\
& \text{\hspace{3.5cm}}\left. K_{D,d}^{n-1}v^{t} f''(p) \ov \adj G_{D,d}(p) \ov{f''}(p)\ov v=1\right\}.
\end{aligned}
\end{equation}
Here $f''(p)$ is the symmetric matrix
\[
\begin{pmatrix}\frac{\pa^2 f}{\pa z_i\pa z_j}(p)\end{pmatrix}_{n \times n},
\]
and $\adj G_{D,d}(p)$ is the adjugate of the matrix $G_{D,d}(p)$. Note that though $I^0$ and $\la^k$ do not depend on $v$, the notations $I_{D,d}^0(p,v)$ and $\la^k_{D,d}(p,v)$ will be used purely for convenience. By Montel's theorem, these infimums are realized. Observe that as $\adj G_{D,d}(p) = g_{D,d}(p) G_{D,d}(p)^{-1}$, we have
\begin{equation}\label{I-M}
\begin{aligned}
I_{D,d}(p,v) & ={K_{D,d}^{n-1}(p)g_{D,d}(p)} M_{D,d}(p,v)\\
&  ={K_{D,d}^{n-1+\frac{1}{d+1}}(p)\be_{D,d}(p)}M_{D,d}(p,v).
\end{aligned}
\end{equation}
If $F: D \to D'$ is a biholomorphism and $\mathcal{I}$ is any of the minimum integrals $I^k$ and $I$, then
\begin{equation}\label{trans-min-int}
\mathcal{I}_{D,d}(p,v)= \big\vert \det F^{\prime}(p)\big\vert^{-2d-2}\mathcal{I}_{D',d}(F(p),F'(p)v).
\end{equation}
Indeed, if $g \in A^2_d(D')$ is a function realising $\mathcal{I}_{D',d}\big(F(p),F'(p)v\big)$, then a routine calculation shows that the function
\[
f(z)=\big(\det F'(p)\big)^{-d-1}\, g\circ F(z) \, \big(\det F'(z)\big)^{d+1},
\]
is a candidate for $\mathcal{I}_{D,d}(p,v)$ and so
\[
\mathcal{I}_{D,d}(p,v) \leq  \big\vert \det F'(p)\big\vert^{-2d-2}\mathcal{I}_{D',d}\big(F(p),F'(p)v\big).
\]
This inequality must hold for $F^{-1}:D'\to D$ as well and hence \eqref{trans-min-int} follows.
% Under a biholomorphism $F: D \to D'$ they transform by the rule
%\begin{equation}\label{trans-min-int}
%I^k_{D,d}(p,v) = \big\vert \det F^{\prime}(p)\big\vert^{-2d-2} I^k_{D',d}\big(F(p), F^{\prime}(p)v\big), \quad k=0,1, 2,3.
%\end{equation}
We also note the homogeneity property
\begin{equation}\label{homogeneity}
\begin{aligned}
I^k_{D,d}(p,v) & = \vert \al\vert^{-2k} I^1_{D,d}(p,v), \quad k=1,2,\\
I_{D,d}(p, \al v) & = \vert \al\vert^{-2} I_{D,d}(p,v), \quad M_{D,d}(p,\al v)  =  \vert \al\vert^{-2} M_{D,d}(p,v).
\end{aligned}
\end{equation}
The minimum integrals $I^k_{D,d}(p,v)$, $k=0,1,2$, are related to $K_{D,d}$, $\tau_{D,d}$, and $R_{D,d}$ by
\begin{equation}\label{metric-min-int}
\begin{aligned}
K_{D, d}(p) & =\frac{1}{I^0_{D,d}(p,v)},\\
\tau^2_{D,d}(p,v)& ={\frac{I^0_{D,d}(p,v)}{I^1_{D,d}(p,v)}}, \text{ and}\\
R_{D,d}(p,v) &= 2- \frac{\big(I^1_{D,d}(p,v)\big)^2}{I^0_{D,d}(p,v) I^2_{D,d}(p,v)}.
\end{aligned}
\end{equation}
These formulas can be derived in the same way as for the usual Bergman kernel and metric  and so we omit the details. In particular, the first relation when combined with
\[
K_{D,d}(p)=\big\|K_{D,d}(\cdot, p)\big\|_{A^2_d(D)}^2,
\]
which is a consequence of the reproducing property, implies that
\[
\left\|\frac{K_{D,d}(\cdot, p)}{K_{D,d}(p)}\right\|^2=\frac{1}{K_{D,d}(p)}=I^0_{D,d}(p,v),
\]
and so $K_{D,d}(\cdot, p)/K_{D,d}(p)$ uniquely realises the minimum integral $I^0_{D,d}(p,v)$. The third relation in \eqref{metric-min-int} implies that $R_{D,d} \leq 2$.

The minimum integrals, $\la_{D,d}^k(p,v)$, $I_{D,d}(p,v)$, and $M_{D,d}(p,v)$ for $d=0$ are due to Krantz-Yau \cite{KY}. More precisely, for the usual Bergman metric, these are the reciprocals of the extremal domain functions in Section~2 of the above mentioned paper which were introduced for localising the Bergman invariant and the Ricci curvature. It is to be noted the notational difference that in the above paper a vector in $\mf{C}^n$ is regarded as a row vector. To obtain their relation to $\be_{D,d}$ and $\Ric_{D,d}$, define
\[
\la_{D,d}(p)=\la_{D,d}^1(p) \la_{D,d}^2(p) \cdots \la_{D,d}^n(p).
\]
Solving the minimization problem associated to $\la^k_{D,d}(p)$ exactly as in the proof of 7(a) in Section 2, Chapter II of \cite{Berg70}, we obtain
\[
\la^{k}_{D,d}(p)=\frac{\det \begin{pmatrix}K_{ij}(p)\end{pmatrix}_{i,j=0}^{k-1}}{\det \begin{pmatrix} K_{ij}(p)\end{pmatrix}_{i,j=0}^{k}},
\]
where
\[
K_{ij}(p)=\frac{\pa^2 K_{D,d}}{\pa z_{i} \pa \ov z_j}(p), \quad 0 \leq i,j\leq n.
\]
Note that
\begin{multline*}
\det\begin{pmatrix} K_{ij}(p)\end{pmatrix}_{i,j=0}^{n}= \det
\begin{pmatrix}
K_{D,d}(p) & \pa K_{D,d}(p)\\
\ov \pa K_{D,d}(p) & \pa \ov \pa K_{D,d}(p)
\end{pmatrix}\\
=K_{D,d}^{n+1}(p) \det \begin{pmatrix}\pa \ov \pa \log K_{D,d} (p)\end{pmatrix}=K_{D,d}^{n+1(p)}g_{D,d}(p).
\end{multline*}
Thus,
\[
\la_{D,d}(p) =\frac{K_{00}(p)}{\det \begin{pmatrix}K_{ij}(p)\end{pmatrix}_{i,j=0}^{1}}\cdots  \frac{\det \begin{pmatrix}K_{ij}(p)\end{pmatrix}_{i,j=0}^{n-1}}{\det \begin{pmatrix}K_{ij}(p)\end{pmatrix}_{i,j=0}^{n}}=\frac{1}{K_{D,d}^{n}(p)g_{D,d}(p)}.
\]
Therefore, using \eqref{metric-min-int}
\begin{equation}\label{g-min-int}
g_{D,d}(p)=\frac{1}{K^n_{D,d}(p)\la_{D,d}(p)} = \frac{\big(I^0_{D,d}(p,v)\big)^n}{\la_{D,d}(p)},
\end{equation}
and
\begin{equation}\label{be-min-int}
\be_{D,d}(p) =\frac{g_{D,d}(p)}{K^{1/(d+1)}_{D,d}(p)}=\frac{1}{K_{D,d}^{n+\frac{1}{d+1}}(p)\la_{D,d} (p,v)} = \frac{\big(I^0_{D,d}(p,v)\big)^{n+\frac{1}{1+d}}}{\la_{D,d} (p,v)}.
\end{equation}
The arguments in the proof of Proposition~2.1 of \cite{KY} applied to $K_{D,d}$ yields
\begin{equation}\label{ric-min-in}
\text{Ric}_{D,d}(p,v) =(n+1)- \frac{1}{K_{D,d}(p) \tau_{D,d}^2(p,v)I_{D,d}(p,v)}.
\end{equation}
Now from \eqref{I-M}, \eqref{metric-min-int}, and \eqref{be-min-int}, we have
\begin{equation}\label{I-M-2}
I_{D,d}(p,v)=\frac{M_{D,d}(p,v)}{K_{D,d}(p)\la_{D,d}(p)}=\frac{I^0_{D,d}(p,v) M_{D,d}(p,v)}{\la_{D,d}(p,v)},
\end{equation}
and hence from \eqref{ric-min-in}, 
\begin{multline}\label{ric-min-int-2}
\text{Ric}_{D,d}(p,v) =(n+1)-\frac{\la_{D,d}(p,v)}{\tau^2_{D,d}(p,v)M_{D,d}(p,v)}\\
=(n+1)-\frac{I^1_{D,d}(p,v) \la_{D,d}(p)}{I^0_{D,d}(p,v) M_{D,d}(p,v)} .
\end{multline}

\subsection{Monotonicity of the minimum integrals}
Let $D\subset D'$, $p \in D$, and $v \in \mf{C}^n$ be a nonzero vector. It is evident from their definitions that
\begin{equation}\label{mon-min-int}
\begin{aligned}
I^k_{D,d}(p,v) & \leq I^k_{D',d}(p,v), \quad k=0,1,2,\\
\la^k_{D,d}(p) & \leq \la^k_{D',d}(p), \quad k=1, \ldots, n.
\end{aligned}
\end{equation}
We note that the minimum integral $I_{D,d}(p,v)$ is not monotonic in general (see the remark on page 236 of \cite{KY}) but $M_{D,d}(p,v)$ is. Indeed, the first and the second equation in \eqref{metric-min-int} gives
\[
K_{D,d}(p)\tau^2_{D,d}(p,v)=\frac{1}{I^1_{D,d}(p,v)},
\]
and hence by \eqref{mon-min-int},
\begin{equation}\label{ns-monotonicity}
K_{D,d}(p)\tau^2_{D,d}(p,v) \geq K_{D',d}(p)\tau^2_{D',d}(p,v).
\end{equation}
This can be written as
\begin{equation}\label{ns-matrix-mon}
K_{D,d}(p) v^{t}G_{D,d}(p)\ov v \geq K_{D',d}(p) v^{t}G_{D',d} \ov v.
\end{equation}
Now the arguments as in the proof of Proposition~2.2 in \cite{KY} gives
\begin{equation}\label{ns-inverse-monotonicity}
\frac{v^{t}\ov G_{D,d}^{-1}(p)\ov v}{K_{D,d}(p)}\leq \frac{v^{t}\ov G_{D',d}^{-1}(p)\ov v}{K_{D',d}(p)},
\end{equation}
and
\begin{equation}\label{mon-M}
M_{D,d}(p,v) \leq M_{D',d}(p,v).
\end{equation}

%Finally, we also define the quantities
%\begin{equation}\label{max-qt}
%\begin{aligned}
%J_{D,d}^0(p,v) & =\sup \left\{\vert f(p)\vert^2 : \|f\|_{A^2_d(D)}^2\leq 1 \right\},\\
%J_{D,d}^1(p,v) & = \sup \left\{\vert f'(p)v\vert^2 : \|f\|_{A^2_d(D)}^2\leq 1, f(p)=0 \right\}, \quad \text{and}\\
%J_{D,d}^2(p,v) & = \sup \left\{\left\vert \sum_{i,j=1}^n \frac{\pa^2 f}{\pa z_i \pa z_j}(p)v^i v^j\right\vert^2: \|f\|_{A^2_d(D)}^2 \leq 1, f(p)=0, f^{\prime}(p)v=0\right\}.
%\end{aligned}
%\end{equation}
%It can be checked that
%\begin{equation}\label{I-J}
%J^k_{D,d}=\frac{1}{I^k_{D,d}}, \quad k=0,1,2.
%\end{equation}
%Thus
%\begin{equation}\label{metric-max-qt}
%K_{D,d}(p)=J^0_{D,d}(p,v), \quad \tau^2_{D,d}(p,v)= {\frac{J^1_{D,d}(p,v)}{J^0_{D,d}(p)}}, \quad R_{D,d}(p,v)=2-\frac{J^0_{D,d}(p,v)J^2_{D,d}(p,v)}{\big(J^1_D(p,v)\big)^2}.
%\end{equation}
%We conclude this section by mentioning that for defining the minimum integrals and establishing their above mentioned properties the boundedness of the domain $D$ is not required as long as their defining families are nonempty. 

%%%%%%%
\subsection{Localization on domains of holomorphy}

\begin{lem}\label{loc-min-int-doh}
Let $D \subset \mf C^n$ be a bounded domain of holomorphy and $U_0 \subset \subset U$ be neighbourhoods of a point $p^0 \in \pa D$. Then there exist constants $c,C>0$ such that if $\mathcal{I}$ is any of the minimum integrals in \eqref{min-int},
\[
c\mathcal{I}_{U \cap D, d}(p,v) \leq \mathcal{I}_{D,d}(p,v) \leq C \mathcal{I}_{U \cap D,d}(p,v)
\]
for all $p \in U_0$ and $v \in \mf{C}^n$.
\end{lem}
\begin{proof}
%First suppose $\mathcal{I} \neq I$. Then the inequality on the left follows from the monotonicity of the minimum integrals noted in \eqref{mon-min-int}. The inequality on the right will follow at once if we prove

We begin with the following:
\begin{claim}There exists a constant $C>0$ such that for any $p \in U_0$ and $f \in A^2_d(U\cap D)$, we can find a $F_p \in A^2_d(D)$ that agrees with $f$ at $p$ up to second derivatives and
%(satisfying $\hat{f}_p(p)=f(p)$, $\hat{f}^{\prime}_p(p)=f^{\prime}(p)$,)% 
\[
\|F_p\|^2_{A^2_d(D)} \leq C \|f\|_{A^2_d(U \cap D)}.
\]
\end{claim}
To prove this claim, we choose a neighbourhood $U_1$ of $p$, where $U_0 \subset\subset U_1 \subset\subset U$, and a cut-off function $\chi \in C_0^{\infty}(U)$, with $0 \leq \chi \leq 1$ and $\chi\equiv 1$ on $U_1$. Let $p \in U_0$ and $f \in A^2_d(U \cap D)$. Define the $(0,1)$ form
\[
\al_f=\ov \pa (\chi f)
\]
which is smooth, $\ov \pa$-closed on $D$, and vanishes on $U_1 \cap D$. Also, consider the plurishubharmonic function $\phi_p$ on $D$ defined by
\[
\phi_p(z)=(2n+6)\log \vert z-p\vert+d\log K_D(z).
\]
Now apply Theorem 4.2 of \cite{Hor} to obtain a solution $u_p$ to the equation $\ov \pa u=\al_f$ on $D$ satisfying
\begin{equation}\label{soln-est-1}
\int_D \frac{\big\vert u_p(z)\big\vert^2}{\vert z-p\vert^{2n+6}(1+\vert z\vert^2)^2} K_D^{-d}(z)\, dV(z) \leq \int_{D} \frac{\vert \al_f\vert^2}{\vert z-p\vert^{2n+6}}K_D^{-d}(z)\,dV(z).
\end{equation}
But $\supp \al_f\subset(U \setminus U_1) \cap D$ and there $\al_f$ satisfies
\[
\vert \al_f \vert^2 \leq c_1 \vert f\vert^2
\]
for some constant $c_1>0$ independent of $p$ and $f$. Moreover there exists a constant $c_2>0$ such that
\[
\vert z-p\vert^{2n+6} \geq c_2
\]
for all $z \in (U \setminus U_1) \cap D$ and $p \in U_0$. Hence, it follows from \eqref{soln-est-1} that
\begin{equation}\label{soln-est-2}
\int_D \frac{\big\vert u_p(z)\big\vert^2}{\vert z-p\vert^{2n+6}(1+\vert z\vert^2)^2} K_D^{-d}(z)\, dV(z) \leq c_3 \|f\|^2_{A^2_d(U\cap D)}
\end{equation}
where $c_3>0$ is a constant independent of $f$ and $p$. On the other hand since $D$ is bounded, there exists a constant $c_4>0$ such that
\[
\vert z-p\vert^{2n+6}(1+\vert z \vert^2)^2\leq c_4
\]
for all $z, p\in D$. Consequently,
\begin{equation}\label{soln-lb-1}
\int_D \frac{\big\vert u_p(z)\big\vert^2}{\vert z-p\vert^{2n+6}(1+\vert z\vert^2)^2} K_D^{-d}(z)\, dV(z) \geq c_5 \|u_p\|^2_{A^2_d(D)}
\end{equation}
where $c_5>0$ is a constant independent of $f$ and $p$. Combining \eqref{soln-est-2} and \eqref{soln-lb-1}, there exists a constant $c_6$ independent of $f$ and $p$, such that
\begin{equation}\label{soln-est-data}
\|u_p\|_{A^2_d(D)} \leq c_6\|f\|_{A^2_d(U\cap D)}.
\end{equation}
Also, note that $u_p$ is holomorphic on $U_1 \cap D$ and  if $\ov B(p,\ep) \subset U_1 \cap D$ then by \eqref{soln-est-2}
\[
\int_{B(p,\ep)} \frac{\vert u_p(z)\vert^2}{\vert z-p\vert^{2n+6}} dV < \infty.
\]
This implies that $u_p$ vanishes at $p$ up to the second derivative. Finally, set $F_p=\chi f-u_p$. Then $F_p$ agrees with $f$ at $p$ up to the second derivatives. Moreover
\[
\|F_p\|_{A^2_d(D)} \leq \|\chi f\|_{A^2_d(D)}+\|u_p\|_{A^2_d(D)} \leq \|f\|_{A^2_d(U\cap D)}+c_6\|f\|_{A^2_d(U\cap D)} %= c\|f\|_{A^2_d(U\cap D)}
\]
using \eqref{soln-est-data}, which proves the claim with $C=(1+c_6)$.

To complete the proof of the lemma, let $p \in U\cap D$, $v \in \mf{C}^n$ be a nonzero vector, and $f \in A^2_d(U\cap D)$ be a minimizing function for $\mathcal{I}_{U\cap D}(p,v)$. Let $C$ and $F_p$ be as given by the claim. Then observe that $F$ is a candidate for $\mathcal{I}_{D}(p,v)$ and so
\begin{equation}\label{mathcal-I-ub}
\mathcal{I}_{D}(p,v) \leq C \mathcal{I}_{U\cap D}(p,v).
\end{equation}
For the reverse inequality, note that if $\mathcal{I} \neq I$, then by the monotonicity of $\mathcal{I}$
\begin{equation}\label{mathcal-I-lb1}
\mathcal{I}_{D,d}(p,v) \geq \mathcal{I}_{U\cap D,d}(p,v)
\end{equation}
and so we can take $c=1$. For $\mathcal{I}=I$, we write $I_{D,d}(p,v)$ in terms of $I^0_{D,d}(p,v)$, $\la_{D,d}(p,v)$, and $M_{D,d}(p,v)$ as given by \eqref{I-M-2}, and then using \eqref{mathcal-I-ub}  and \eqref{mathcal-I-lb1} for the latter minimum integrals,
\[
I_{D,d}(p,v) \geq \frac{I^0_{U\cap D,d} (p,v)M_{U\cap D,d}(p,v)}{C\la_{U\cap D}(p,v)}=\frac{1}{C} I_{U\cap D, d}(p,v).
\]
This completes the proof.
\end{proof}

In view of \eqref{metric-min-int}, \eqref{g-min-int}, \eqref{be-min-int}, and \eqref{ric-min-int-2}, this lemma immediately gives

\begin{thm}\label{localise-doh}
Let $D \subset \mf C^n$ be a bounded domain of holomorphy and $U_0 \subset \subset U$ be neighbourhoods of a point $p^0 \in \pa D$. Then there exist constants $c,C>0$ such that
\begin{equation*}
\begin{aligned}
& cK_{U \cap D, d}(p) \leq K_{D, d}(p) \leq  K_{U\cap D, d}(p),\\
& cg_{U \cap D, d}(p) \leq g_{D, d}(p) \leq  g_{U\cap D, d}(p),\\
& c\be_{U \cap D, d}(p) \leq \be_{D, d}(p) \leq  \be_{U\cap D, d}(p),\\
& c\tau_{U \cap D,d}(p,v) \leq \tau_{D,d}(p,v) \leq C \tau_{U \cap D, d}(p,v),\\
& cR_{U \cap D,d}(p,v) \leq 2-R_{D,d}(z,v) \leq C R_{U \cap D, d}(p,v),\\
& c\Big((n+1)-\Ric_{U \cap D,d}(p,v)\Big) \leq n+1-R_{D,d}(z,v)\\
&\text{\hspace{5cm}}\leq C \Big((n+1)-\Ric_{U \cap D, d}(p,v)\Big),
\end{aligned}
\end{equation*}
for all $p \in U_0$ and $v \in \mf{C}^n$.
\end{thm}

%%%%%%%%
\subsection{Localization near peak points}

While the localization of the weighted kernel and the Narasimhan--Simha type metric in Theorem~\ref{localise-doh} are in terms of small--large--constants, more precise estimates can be given near a holomorphic peak point as was observed for the Bergman kernel and metric in \cite{KY}. 
\begin{lem}\label{loc-min-int-peak}
Let $D$ be a bounded pseudoconvex domain and $p^0 \in \pa D$ be a local holomorphic peak point of $D$. Then for a sufficiently small neighbourhood $U$ of $p^0$,
\[
\lim_{p \to p^0}\frac{\mathcal{I}_{U \cap D, d}(p,v)}{\mathcal{I}_{D,d}(p,v)}=1,
\]
where $\mathcal{I}$ is any of the minimum integrals in \eqref{min-int}. Moreover, the convergence is uniform on $\big\{ \Vert v \Vert = 1\big\}$.
\end{lem}

\begin{proof}
Let $h$ be a local holomorphic peak function at $p^0$ defined in a neighbourhood $U$ of $p^0$. Take any neighbourhood $U_0$ of $p^0$ such that $U_0 \subset\subset U$ and $h$ is nonvanishing on $U_0$.
\begin{claim}
There exist constants $0<a<1$ and $c>0$ depending only on $U$ with the following property: given any $p\in U_0$, a function $f \in A^2_d(U \cap D)$, and an integer $N\geq 1$, there exists $F_{p,N} \in A^2_d(D)$ satisfying
\begin{equation}\label{F-est}
\|F_{p,N}\|_{A^2_d(D)}  \leq \frac{1+ ca^N}{\big\vert h(p)\big\vert^N} \|f\|_{A^2_d(U \cap D)},
\end{equation}
and such that (i) $F_{p,N}(p)=f(p)$, (ii) $F'_{p,N}(p)=f'(p)$ if $f(p)=0$, (iii) $v^{T}F_{p,N}''(p)v=v^{T}f''(p)v$ if $f(p)=0$ and $f'(p)v=0$ where $v \in \mf{C}^n$, and (iv) $F_{p,N}''(p)=f''(p)$ if $f(p)=0$ and $f'(p)=0$.
\end{claim}
To prove this claim, choose a neighbourhood $U_1$ of $p^0$ such that $U_0 \subset\subset U_1\subset\subset U$. Then there is a constant $a\in (0,1)$ such that $\vert h \vert <a$ on $\ov {\big((U \setminus U_1) \cap D\big)}$. Choose a cut-off function $\chi \in C_0^{\infty}(U)$ satisfying $0 \leq \chi \leq 1$ on $U$, and $\chi \equiv1$ on $U_1$. Let $p \in U_0$ and $f \in A^2_d(U \cap D)$. Define the $(0,1)$ form
\[
\al_f=\ov \pa(\chi f h^N),
\]
which is smooth and $\ov \pa$-closed on $D$. Moreover, $\supp \al_f \subset (U \setminus U_1) \cap D$ and there $\al_f$ it satisfies
\[
\vert \al_f \vert^2 = \vert \ov \pa \chi \vert^2 \vert f \vert^2 \vert h\vert^2 \leq c_1a^{2N} \vert f\vert^2,
\]
for some constant $c_1>0$ independent of $p$ and $f$. Also, consider the plurishubharmonic function $\phi_p$ on $D$ defined in Lemma~\ref{loc-min-int-doh} and solve the equation $\ov \pa u=\al$ with $L^2$-estimates as in there. Repeating the same arguments in that lemma, we arrive at a solution $u_p$ which vanishes at $p$ up to the second derivatives and which satisfies
\begin{equation}\label{peak-soln-est-data}
\|u_p\|_{A^2_d(D)} \leq c \vert a \vert^{N}\|f\|_{A^2_d(U\cap D)},
\end{equation}
where $c$ is a constant independent of $p$ and $f$. Define
\begin{equation}\label{F-defn}
F_{p,N}=\frac{\chi f h^N-u_p}{h^N(p)}.
\end{equation}
Then $F_{p,N}$ belongs to $A^2_d(D)$ and 
\begin{align*}
\big\vert h(p)\big\vert^N\|F_{p,N}\|_{A^2_d(D)} \leq \|\chi f h^N\|_{L^2_d(D)} + \|u_p\|_{L^2_d(D)}
 \leq \|f\|_{A^2_d(U \cap D)}+ ca^N \|f\|_{A^2_d(U \cap D)},
\end{align*}
using \eqref{peak-soln-est-data} and thus $F_{p,N}$ satisfies the estimate \eqref{F-est}. We now proceed to show that $F_{p,N}$ satisfies the other properties in our claim. First note that (i) follows from \eqref{F-defn}. Now differentiating $F_{p,N}$ in $U_0$ and as $\chi=1$ there, we obtain
\begin{equation*}%\label{1-der-F}
\frac{\pa F_{p,N}}{\pa z_j}=\frac{1}{h^N(p)} \left\{\frac{\pa f}{\pa z_j} h^N+ N f h^{N-1} \frac{\pa h}{\pa z_j} -\frac{\pa u_p}{\pa z_j}\right\}, \quad j=1, \ldots n,
\end{equation*}
from which (ii) follows. Differentiating the above equation, we obtain that if $f(p)=0$, then
\begin{multline*}
\frac{\pa^2 F_{p,N}}{\pa z_i \pa z_j}(p)= \frac{1}{h^N(p)} \left\{h^N(p)\frac{\pa^2f}{\pa z_i\pa z_j}(p)  + N h^{N-1}(p)\frac{\pa f}{\pa z_j}(p) \frac{\pa h}{\pa z_i}(p)\right.\\
\left.+N h^{N-1}(p)\frac{\pa f}{\pa z_i}(p) \frac{\pa h}{\pa z_j}(p)\right\},
\end{multline*}
for $1 \leq i,j\leq n$, from which (iii) and (iv) follows. This proves our claim.

Now to complete the proof of the lemma, fix $p\in U_0$, $v\in \mf{C}^n \setminus\{0\}$, and let $f$ be a minimizing function for $\mathcal{I}_{U \cap D}(p,v)$. Let the constants $c,a$, and the function $F_{p,N}$ be as given by the claim. Then observe that $F_{p,N}$ is a candidate for $\mathcal{I}_{ D}(p,v)$ and so \eqref{F-est} gives
\begin{equation}\label{I-ub}
\frac{\mathcal{I}_{D}(p,v)}{\mathcal{I}_{U \cap D}(p,v)} \leq \frac{(1+ca^{N})^2}{\big\vert h(p)\big\vert^{2N}}.
\end{equation}
First, let $p \to p^0$ to get
\[
\limsup_{p \to p^0}\frac{\mathcal{I}_{D,d}(p,v)}{ \mathcal{I}_{U \cap D, d}(p,v)} \leq (1+ca^{N})^2,
\]
and then let $N \to \infty$ to get
\begin{equation}\label{ulimit-mathcal-I}
\limsup_{p \to p^0}\frac{\mathcal{I}_{D,d}(p,v)}{\mathcal{I}_{U \cap D, d}(p,v)} \leq 1.
\end{equation}

Now suppose $\mathcal{I}\neq I$. Then by the monotonicity of $\mathcal{I}$, we have $\mathcal{I}_{D,d} \geq \mathcal{I}_{U\cap D,d}$, and so
\begin{equation}\label{llimit-mathcal-I}
\liminf_{p\to p^0} \frac{\mathcal{I}_{D,d}(p,v)}{\mathcal{I}_{U \cap D, d}(p,v)} \geq 1.
\end{equation}
Combining \eqref{ulimit-mathcal-I} and \eqref{llimit-mathcal-I}, we obtain
\begin{equation}\label{loc-mathcal-I}
\lim_{p \to p^0} \frac{\mathcal{I}_{D,d}(p,v)}{\mathcal{I}_{U \cap D, d}(p,v)}=1.
\end{equation}
Note that the right hand side of \eqref{I-ub} is independent of $v$ and hence the convergence is uniform in unit vectors $v \in \mf{C}^n$.

Now let $\mathcal{I}=I$. We express $I$ in terms of $I^0$, $\la$, and $M$ as in \eqref{I-M-2}, and then using monotonicity of $I^0$ and $M$ we obtain
\[
\frac{I_{D,d}(p,v)}{I_{U \cap D, d}(p,v)} \geq \frac{\la_{U\cap D,d}(p)}{\la_{D,d}(p)}.
\]
Applying \eqref{llimit-mathcal-I} to $\la$,
\begin{equation}\label{llimit-I}
\liminf_{p \to p^0} \frac{I_{D,d}(p,v)}{I_{U \cap D, d}(p,v)} \geq 1.
\end{equation}
Combining this with \eqref{ulimit-mathcal-I} and \eqref{llimit-I}, we obtain
\[
\lim_{p \to p^0} \frac{I_{D,d}(p,v)}{I_{U \cap D, d}(p,v)}=1.
\]
Again, as the right hand side of \eqref{I-ub} is independent of $v$, the convergence is uniform in unit vectors $v \in \mf{C}^n$. This completes the proof of the lemma.
\end{proof}

In view of \eqref{metric-min-int}, \eqref{g-min-int}, \eqref{be-min-int}, and \eqref{ric-min-int-2}, we have

\begin{thm}\label{localise-peak}
Let $D\subset \mf{C}^n$ be a bounded pseudoconvex domain and $p^0 \in \pa D$ be a local holomorphic peak point of $D$. Then for a sufficiently small neighbourhood $U$ of $p^0$,
\begin{align*}
\lim_{p \to p^0}\frac{K_{U \cap D, d}(p)}{K_{D,d}(p)}=1, \quad \lim_{p \to p^0}\frac{g_{U \cap D, d}(p)}{g_{D,d}(p)}, \quad \lim_{p \to p^0} \frac{\be_{U\cap D,d}(p)}{\be_{D,d}(p)}=1,
\end{align*}
and also
\begin{align*}
& \lim_{p \to p^0}\frac{\tau_{U \cap D,d}(p,v)}{\tau_{D,d}(p,v)}=1,  \quad  \lim_{p \to p^0} \frac{2-R_{U \cap D}(p,v)}{2-R_{D}(p,v)}=1,\quad \text{and}\\ 
& \lim_{p \to p^0} \frac{n+1-\Ric_{U \cap D}(p,v)}{n+1-\Ric_{D}(p,v)}=1,
\end{align*}
uniformly on $\{ \Vert v \Vert = 1\}$.
\end{thm}

%%%%%
\section{A Ramadanov type theorem}
In this section, we prove a stability result for the weighted kernel $K_{D, d}$. This is an analog of a result that were recently obtained for the Bergman kernel in \cites{BBMV1}. Let $D$ be a domain in $\mf{C}^n$. For $q \in \mf{C}^n$, denote by $D-q$ the image of $D$ under the translation $Tv=v-q$ and for $r>0$, by $rD$ the image of $D$ under the homothety $Rv=rv$.
\begin{prop}\label{RT}
Let $D_j$ be a sequence of domains in $\mf{C}^n$ converging to a domain $D$ in $\mf{C}^n$ in the following way:
\begin{enumerate}
\item[{\em (i)}] any compact subset of $D$ is eventually contained in each $D_j$,

\item[{\em (ii)}] there exists a common interior point $q$ of $D$ and all $D_j$, such that for every $\ep>0$ there exists $j_{\ep}$ satisfying
\[
D_{j}-q \subset (1+\ep)(D-q),\]
for all $j \geq j_{\ep}$.
\end{enumerate}
Assume further that $D$ is star-convex with respect to $q$, and both $K_D, K_{D,d}$ are nonvanishing along the diagonal. Then $K_{D_j,d} \to K_{D,d}$ uniformly on compact subsets of $D \times D$ together with all the derivatives.
\end{prop}
\begin{proof}
Without loss of generality we may assume that $q=0$. Let $\Om$ be a relatively compact subdomain of $D$. Then $\Om \subset D_j$ for all large $j$. So
\begin{equation}\label{K-j-diag-bound}
K_{D_j,d}(z) \leq K_{\Om}(z)
\end{equation}
for $z \in \Om$ and $j$ large. Also, note that
\begin{equation}\label{K-j-bound}
\big\vert K_{D_j,d}(z,w)\big\vert \leq \sqrt{K_{D_j,d}(z)} \sqrt{K_{D_j,d}(w)}
\end{equation}
for $z, w \in D_j$ and for each $j$. It follows from \eqref{K-j-diag-bound} and \eqref{K-j-bound} that $K_{D_j,d}(z,w)$ is locally uniformly bounded and hence has a subsequence that converges locally uniformly to a function, say, $K_{\infty,d}(z,w)$ on $D \times D$. Note further that by (ii) in the hypothesis of the proposition, we have for all sufficiently large $j$, $D_j \subset 2D$ so that
\[
K_{D_j,d}(p) \geq K_{2D,d}(p)>0,
\]
where the last inequality follows from the fact that $2D$ is biholomorphic to $D$, and $K_{D,d}$ is given to be nonvanishing along the diagonal, and consequently,
\[
K_{\infty,d}(p) \geq K_{2D,d}(p)>0.
\] 

\medskip

Next we show that $K_{\infty,d}(z,w)=K_{D,d}(z,w)$ using the following minimizing property from Section~3: for $p \in D$, the function $K_{D,d}(\cdot, p)/K_{D,d}(p)$ realises the minimum integral $I^0_{D,d}(p,v)$, i.e., it uniquely solves the extremal problem
\[
\text{minimize}\, \|f\|^2_{A^2_d(D)} \text{ subject to } f \in A^2_d(D), f(p)=1.
\]
So fix a point $p\in D$ and let $f\in A^2_d(D)$ with $f(p)=1$. To derive a relation between $K_{\infty}$ and $f$, observe that by (ii) in the hypothesis of the proposition, after passing to a subsequence if necessary, we have for each $j$, $D_j \subset (1+j^{-1})D$ and so by monotonicity,
\begin{equation}\label{KD_j-KD_1+1/j}
K_{(1+j^{-1})D,d}(p) \leq K_{D_j,d}(p).
\end{equation}
Now define the function $g_j$ on $(1+j^{-1})D$ by
\begin{equation}\label{g_j-defn}
g_j(w)=\frac{f\left(\frac{w}{1+j^{-1}}\right)}{f\left(\frac{p}{1+j^{-1}}\right)}
\end{equation}
for $w \in (1+j^{-1})D$. Note that if $j$ is sufficiently large then $f\left(\frac{p}{1+j^{-1}}\right)\neq 0$ by the continuity of $f$ and so $g_j$ is well-defined. Moreover, $g_j(p)=1$ and so
\begin{multline}\label{KD_1+1/j-f}
\frac{1}{K_{(1+j^{-1})D,d}(p)}  \leq \int_{(1+j^{-1})D} \big\vert g_j(w) \big\vert^2\, K_{(1+j^{-1})D}^{-d}(w) \,dV(w)\\
%& =\int_{D} \Big\vert g_j\big((1+j^{-1})z\big)\Big\vert^2 K^{-d}_{(1+j^{-1})D}\big((1+j^{-1})z\big)\,(1+j^{-1})^{2n}\,dV(z)\\
%& =\int_{D} \frac{\vert f(z)\vert^2}{\left\vert f\left(\frac{p}{1+j^{-1}}\right)\right\vert^2} K_{D}^{-d}(z) (1+j^{-1})^{2dn} (1+j^{-1})^{2n} \,dV(z)\\& =
=\frac{(1+j^{-1})^{2n(d+1)}}{\left\vert f\left(\frac{p}{1+j^{-1}}\right)\right\vert^2}  \int_{D} \big\vert f(z)\big\vert^2 \,K_D^{-d}(z)\, dV(z),
\end{multline}
by substituting $w=(1+j^{-1})z$, and using $dV(w)=(1+j^{-1})^{2n}dV(z)$ and the transformation rule
\[
K_D(z) = (1+j^{-1})^{2n} K_{(1+j^{-1})D}((1+j^{-1})z).
\]
By \eqref{KD_j-KD_1+1/j} and \eqref{KD_1+1/j-f}, we have
\begin{multline}\label{KD_j-f}
\frac{1}{K_{\infty}(p)} = \liminf_{j \to \infty} \frac{1}{K_{D_j,d}(p)} \leq \liminf_{j \to \infty} \frac{1}{K_{(1+j^{-1})D,d}(p)}\\ =  \int_{D} \big\vert f(z)\big\vert^2 \,K_D^{-d}(z)\, dV(z).
\end{multline}
We now consider two cases:

Case I. $d=0$: By Fatou's lemma
\begin{multline}\label{int-K-ub}
\int_\Om  \left\vert \frac{K_{\infty}(z,p)}{K_{\infty}(p)} \right\vert^2 \, dV(z)
\leq \liminf_{j \to \infty} \int_{\Om} \left\vert \frac{K_{D_j}(z,p)}{K_{D_j}(p)}\right\vert^2 \, dV(z)\\ 
\leq \liminf_{j \to \infty} \int_{D_j} \left\vert \frac{K_{D_j}(z,p)}{K_{D_j}(p)}\right\vert^2 \, dV(z).
\end{multline}
Since $K_{D_j}(\cdot, z)$ reproduces functions of $A^2(D_j)$,
\begin{equation}\label{KD_j-1/KD_j}
\int_{D_j} \left\vert \frac{K_{D_j}(z,p)}{K_{D_j}(p)}\right\vert^2 \, dV(z) = \frac{1}{K_{D_j}(p)}.
\end{equation}
Combining \eqref{int-K-ub} and \eqref{KD_j-1/KD_j} with \eqref{KD_j-f}, we obtain
\begin{equation}\label{int-K_j-ub}
\int_\Om  \left\vert \frac{K_{\infty}(z,p)}{K_{\infty}(p)} \right\vert^2 \, dV(z) \leq \frac{1}{K_{\infty}(p)} \leq  \int_{D} \big\vert f(z)\big\vert^2 \, dV(z).
\end{equation}
Since $\Om$ is arbitrary, this implies that
\begin{equation}\label{K-infty-f}
\int_D  \left\vert \frac{K_{\infty}(z,p)}{K_{\infty}(p)} \right\vert^2 \, dV(z)  \leq \frac{1}{K_{\infty}(p)} \leq \int_D \big\vert f(z)\big\vert^2 \, dV(z).
\end{equation}
By the minimising property we must have
\begin{equation}\label{K_infty-K_D-ratio}
\frac{K_{\infty}(z,p)}{K_{\infty}(p)}=\frac{K_{D}(z,p)}{K_{D}(p)}, \quad z \in D,
\end{equation}
and also by taking infimum over $f$ in \eqref{K-infty-f}
\begin{equation}\label{K-infty-K_D}
\int_D  \left\vert \frac{K_{\infty}(z,p)}{K_{\infty}(p)} \right\vert^2 \, dV(z)  \leq \frac{1}{K_{\infty}(p)} \leq \frac{1}{K_D(p)} = \int_D  \left\vert \frac{K_{D}(z,p)}{K_{D}(p)} \right\vert^2 \, dV(z).
\end{equation}
By \eqref{K_infty-K_D-ratio}, all the inequalities in \eqref {K-infty-K_D} are equalities and hence $K_{\infty}(p)=K_{D}(p)$. This in turn implies from \eqref{K_infty-K_D-ratio} that $K_{\infty}(z,p)=K_D(z,p)$ for all $z \in D$. The above argument also shows that any convergent subsequence of $K_{D_j}(z,w)$ has limit $K_{D}(z,w)$ and hence $K_{D_j}(z,w)$ itself converges to $K_{D}(z,w)$. Convergence of derivatives follows from the fact that the functions $K_{D_j}(z,w)$ are harmonic.

\medskip

Case II. $d \geq 1$: Note that by the previous case, the weights $K_{D_j}^{-d}(z) \to K_{D}^{-d}(z)$ locally uniformly on $D$. We now repeat (for clarity) the arguments from Case I. By Fatou's lemma,
\begin{equation}\label{int-K-ub-d}
\begin{aligned}
& \int_\Om  \left\vert \frac{K_{\infty}(z,p)}{K_{\infty}(p)} \right\vert^2  K_{D}^{-d}(z) \, dV(z)\\
\leq & \liminf_{j \to \infty}  \int_{\Om} \left\vert \frac{K_{D_j,d}(z,p)}{K_{D_j,d}(p)}\right\vert^2 K_{D_j}^{-d}(z) \, dV(z)\\
\leq & \liminf_{j \to \infty} \int_{D_j} \left\vert \frac{K_{D_j,d}(z,p)}{K_{D_j,d}(p)}\right\vert^2 K_{D_j}^{-d}(z)\,dV(z).
\end{aligned}
\end{equation}
Since $K_{D_j,d}(\cdot, z)$ reproduces the functions in $A_d^2(D_j)$,
\begin{equation}\label{KD_j-1/KD_j-d}
\int_{D_j} \left\vert \frac{K_{D_j,d}(z,p)}{K_{D_j,d}(p)}\right\vert^2 \, K_{D_j}^{-d}(z)\,dV(z) = \frac{1}{K_{D_j,d}(p)}.
\end{equation}
Combining \eqref{int-K-ub-d} and \eqref{KD_j-1/KD_j-d} with \eqref{KD_j-f}, we obtain
\begin{equation}\label{int-K_j-ub-d}
\int_\Om  \left\vert \frac{K_{\infty}(z,p)}{K_{\infty}(p)} \right\vert^2 \, K_{D}^{-d}(z)\,dV(z) \leq \frac{1}{K_{\infty}(p)} \leq \int_{D} \big\vert f(z)\big\vert^2 \,K_D^{-d}(z)\, dV(z).
\end{equation}
Since $\Om$ is arbitrary, this implies that
\begin{equation}\label{K-infty-f-d}
\int_D  \left\vert \frac{K_{\infty}(z,p)}{K_{\infty}(p)} \right\vert^2 \, K_{D}^{-d}(z) \,dV(z)  \leq \frac{1}{K_{\infty}(p)} \leq \int_D \big\vert f(z)\big\vert^2 \, K_{D}^{-d}(z) \,dV(z).
\end{equation}
By the minimizing property we must have
\begin{equation}\label{K_infty-K_D-ratio-d}
\frac{K_{\infty}(z,p)}{K_{\infty}(p)}=\frac{K_{D,d}(z,p)}{K_{D,d}(p)}, \quad z \in D,
\end{equation}
and also by taking infimum over $f$ in \eqref{K-infty-f-d}
\begin{multline}\label{K-infty-K_D-d}
\int_D  \left\vert \frac{K_{\infty}(z,p)}{K_{\infty}(p)} \right\vert^2 \, K_{D}^{-d}(z)\,dV(z)  \leq \frac{1}{K_{\infty}(p)} \\ \leq \frac{1}{K_{D,d}(p)} = \int_D  \left\vert \frac{K_{D,d}(z,p)}{K_{D,d}(p)} \right\vert^2 \, dV(z).
\end{multline}
By \eqref{K_infty-K_D-ratio-d}, all the inequalities in \eqref {K-infty-K_D-d} are equalities and hence $K_{\infty}(p)=K_{D,d}(p)$. This in turn implies from \eqref{K_infty-K_D-ratio-d} that $K_{\infty}(z,p)=K_{D,d}(z,p)$ for all $z \in D$. The above argument also shows that any convergent subsequence of $K_{D_j,d}(z,w)$ has limit $K_{D,d}(z,w)$ and hence $K_{D_j,d}(z,w)$ itself converges to $K_{D,d}(z,w)$. Convergence of derivatives follows from the fact that the functions $K_{D_j,d}(z,w)$ are harmonic.
\end{proof}

%%%%%%%
\section{Proof of Theorem~\ref{spscvx}}

\subsection{Change of coordinates}
We fix some notations first. For a $C^1$-smooth real-valued function $\rho$ defined on some open subset $U \subset \mf{C}^n$, we denote by $\nabla \rho$ the gradient vector field of $\rho$, by $\nabla_{z}\rho$ the vector field $(\pa \rho/\pa z_1, \ldots, \pa \rho/\pa z_n)$, and $\nabla_{\ov z} \rho=\ov{\nabla_{z}\rho}$. The identity map of $\mf{C}^n$ will be denoted by $I$. For any linear map $L: \mf{C}^n \to \mf{C}^n$, its matrix will be denoted by $\mf{L}$. The following change of coordinates was introduced by Pinchuk -- see Lemma~2.1 in \cite{Pin}.

\begin{lem}\label{pin}
Let $D\subset \mf{C}^n$ be a $C^2$-smoothly bounded strongly pseudoconvex domain, $p^0 \in \pa D$, and $\rho$ a $C^2$-smooth local defining function for $D$ defined in a neighbourhood $U$ of $p^0$. Assume further that $\nabla_{z} \rho(p^0)=({'}0,1)$ and $(\pa \rho/\pa z_n)(z)\neq 0$ for all $z \in U$. Then, there is a family of biholomorphic maps $h_{\zeta}:\mf{C}^n \to \mf{C}^n$ depending continuously on $\zeta \in U \cap \pa D$ that satisfies
\begin{enumerate}[{\em (i)}]
\item $h_{p^0}=I$.
\item $h_{\zeta}(\zeta)=0$.
\item The local defining function $\rho_{\zeta}=\rho \circ h_{\zeta}^{-1}$ of the domain $D_{\zeta}:=h_{\zeta}(D)$, near the origin, has the form
\[
\rho_{\zeta}(z)=2 \Re \big(z_n + Q_{\zeta}(z) \big) + H_{\zeta}(z) + R_{\z}(z),
\]
where $Q_{\zeta}(z) = \sum_{\mu, \nu=1}^n a_{\mu \nu} (\zeta) z_{\mu} z_{\nu} $, $ H_{\zeta}(z) = \sum_{\mu, \nu=1}^n b_{\mu \ov \nu} (\zeta) z_{\mu} \overline{z}_{\nu} $ with
$ Q_{\zeta}('z, 0) \equiv 0 $, $ H_{\zeta}('z, 0) \equiv |'z|^2 $, and $R_{\z}(z)=o(\vert z\vert^2)$. The functions $a_{\mu \nu} (\zeta)$, $b_{\mu \ov \nu} (\zeta)$, and $R_{\z}(z)$ depend continuously on $\z$.
\item The mapping $h_{\z}$ takes the real normal $n_{\z}=\{z \in \mf{C}^n: z=\z+2t\nabla_{\ov z}\rho(\z), t \in \mf{R}\}$ to $\pa D$ at $\z$ into the real normal $\{z \in \mf{C}^n: {'z}=y_n=0\}$ to $\pa D_{\z}$ at the origin.
\end{enumerate}
\end{lem}
We will require the {\it derivative} of the map $h_{\z}$, and so let us quickly recall its construction. For each $\z \in U \cap \pa D$, the map $h_{\z}$ is defined as the composition $h_{\z}=\phi^{\z}_3 \circ \phi^{\z}_2 \circ \phi^{\z}_1$ where the maps $\phi^{\z}_i$ are described below. Fix $\zeta \in U \cap \pa D$.

\medskip

\begin{itemize}
\item The defining function $\rho$ near $\z$ has the form
\begin{multline}\label{rho-0}
\rho(z)=2\Re\left(\sum_{\mu=1}^n \frac{\pa \rho}{\pa z_{\mu}}(\z)(z_{\mu}-\z_{\mu}) + \frac{1}{2} \sum_{\mu,\nu=1}^n \frac{\pa^2\rho}{\pa z_{\mu} \pa z_{\nu}}(\z) (z_{\mu}-\z_{\mu})(z_{\nu}-\z_{\nu}) \right)\\+\sum_{\mu,\nu=1}^n \frac{\pa^2 \rho}{\pa z_{\mu}\pa \ov z_{\nu}}(\z) (z_{\mu}-\z_{\mu})(\ov z_{\nu}-\ov \z_{\nu}) +o(\vert z-\z\vert^2).
\end{multline}
Define $w=\phi^{\z}_1(z)$ by
\begin{equation}\label{phi1}
\begin{aligned}
w_{\nu} & =\frac{\pa \rho}{\pa \ov z_n}(\zeta)(z_{\nu}-\zeta_{\nu}) - \frac{\pa \rho}{\pa \ov z_{\nu}}(\zeta)(z_n-\zeta_n), \quad \nu=1, \ldots, n-1,\\
w_n & =\sum_{\nu=1}^{n} \frac{\pa \rho}{\pa z_\nu}(\z)(z_{\nu}-\zeta_{\nu}),
\end{aligned}
\end{equation}
i.e., $w=P_{\z}(z-\z)$ where $P_{\z}$ is the linear map whose matrix is
\begin{align}\label{P-matrix}
\mf{P}_{\z}=\begin{pmatrix}
\frac{\pa \rho}{\pa \ov z_n}(\z) & 0 & \cdots & 0 & -\frac{\pa \rho}{\pa \ov z_1}(\z)\\
0 & \frac{\pa \rho}{\pa \ov z_n}(\z) & \cdots & 0 & -\frac{\pa \rho}{\pa \ov z_2}(\z)\\
\vdots & \vdots & \cdots &\vdots & \vdots\\
0 & 0 & \cdots & \frac{\pa \rho}{\pa \ov z_n}(\z) & -\frac{\pa \rho}{\pa z_{n-1}}(\z)\\
\frac{\pa \rho}{\pa z_1}(\z) & \frac{\pa \rho}{\pa z_2}(\z) & \cdots & \frac{\pa \rho}{\pa z_{n-1}}(\z) & \frac{\pa \rho}{\pa z_n}(\z)
\end{pmatrix},
\end{align}
and let $D^{\z}_1=\phi^{\z}_1 (D)$. Observe that $\phi_1^{\z}(\z)=0$ and
\begin{equation}\label{phi1-n_z}
\phi_1^{\z} \big(\z+t\nabla_{\ov z}\rho(\z)\big) =tP_{\z} \nabla_{\ov z} \rho(\z)=\big({'0}, t\vert \nabla_{\ov z}\rho(\z)\vert^2\big).
\end{equation}
Thus $\phi_1^{\z}$ maps $n_\z$ to the $\Re z_n$-axis. The latter is the real normal to $\pa D^{\z}_1$ at $0$ which can be seen from the Taylor series expansion of the defining function $\rho^{\z}_1=\rho \circ (\phi_1^{\z})^{-1}$ for $\pa D^{\z}_1$ at $0$. Indeed, by writing \eqref{rho-0} in terms of $w$ coordinates, and then replacing $w$ by $z$ itself, $\rho^{\z}$ has the form
\begin{equation}\label{df-1}
\rho_1^{\z}(z)=2\Re \Big(z_n+Q_1^{\z}(z)\Big) +H_1^{\zeta}(z) + \al_1^{\zeta}(z)
\end{equation}
near the origin, where
\begin{equation}\label{a1-b1}
\begin{aligned}
Q_1^{\z}(z) &=\sum_{\mu,\nu=1}^{n} a^1_{\mu \nu}(\zeta) z_{\mu} z_{\nu}, \quad \begin{pmatrix}a^1_{\mu \nu}(\z)\end{pmatrix}=\frac{1}{2}(\mf{P}_{\z}^{-1})^t \begin{pmatrix} \frac{\pa^2\rho}{\pa z_{\mu} \pa z_{\nu}}(\z) \end{pmatrix}\mf{P}_{\z}^{-1}\\
H_1^{\z}(z) & =\sum_{\mu,\nu=1}^n b^1_{\mu \ov \nu}(\z) z_{\mu} \ov z_{\nu}, \quad 
\begin{pmatrix}b^1_{\mu\ov \nu}(\z)\end{pmatrix}=(\mf{P}_{\z}^{-1})^{*}\begin{pmatrix} \frac{\pa^2 \rho}{\pa z_{\mu}\pa \ov z_{\nu}}(\z) \end{pmatrix}\mf{P}_{\z}^{-1},
\end{aligned}
\end{equation}
and $\al_1^{\z}(z) =o(\vert z\vert^2)$. It follows from \eqref{df-1} that the real normal to $\pa D^{\z}_1$ at $0$ is $\Re z_n$-axis. It is evident from \eqref{phi1} that as $\z \to \z^0$, $\phi_1^{\z}(z) \to \phi_1^{\z^0}(z)$ uniformly on compact subsets of $\mf{C}^n$. Moreover, $(\phi_1^{\z})'(z)=\mf{P}_{\z}$ for all $z \in \mf{C}^n$. Therefore, as $\z \to \z^0$, we also have $(\phi^{\z}_1)'(z)\to (\phi^{\z^0}_1)'(z)$ in the operator norm and uniformly in $z \in \mf{C}^n$.

\medskip 

\item The transformation $w=\phi^{\z}_2(z)$ is a polynomial automorphism defined by
\begin{equation}\label{phi2}
w=\Big({'z}, z_n+ \sum_{\mu,\nu=1}^{n-1} a^1_{\mu \nu}(\zeta) z_{\mu} z_{\nu}\Big).
\end{equation}
In particular, note that $\phi_2^{\z}$ fixes the points on the $\Re z_n$-axis.
In these new coordinates $w$, which we denote by $z$ itself, the defining function $\rho^{\z}_2=\rho\circ (\phi_1^\z)^{-1}\circ (\phi_2^{\z})^{-1}$ of the domain $\phi_2^{\z}\circ \phi_1^{\z}(D)$, near the origin, has the form
\begin{equation}\label{fm-2}
\rho_2^{\z}(z)=2\Re \Big(z_n+Q_2^{\z}(z)\Big) + H_2^{\zeta}(z) + \al_2^{\zeta}(z),
\end{equation}
where $Q_2^{\z}(z) =\sum_{\mu,\nu=1}^{n} a^2_{\mu \nu}(\zeta) z_{\mu} z_{\nu}$ and $H_2^{\z}(z) =\sum_{\mu,\nu=1}^n b^2_{\mu \ov \nu}(\z) z_{\mu} \ov z_{\nu}$ with
\begin{align*}
& a^2_{\mu \nu}(\z)=0 \text{ for } 1 \leq \mu,\nu \leq n-1 \text{ and } a^1_{\mu \nu}(\z)\text{ otherwise.}\\
&  b^2_{\mu\ov \nu}(\z)=b^1_{\mu \ov \nu}(\z) \text{ for } 1 \leq \mu,\nu \leq n.
\end{align*}
Equation \eqref{a1-b1} implies that $a_{\mu\nu}^1(\z)$ depends continuously on $\z$ and hence it follows from \eqref{phi2} that as $\z \to \z^0$, we have $\phi_2^{\z}(z) \to \phi_2^{\z^0}(z)$ uniformly on compact subsets of $\mf{C}^n$. Also,
\begin{equation}\label{der-phi2}
(\phi_2^{\z})'(z)= \begin{pmatrix}\mbb{I}_{n-1} & 0\\
\begin{pmatrix}\displaystyle \sum_{\mu=1}^{n-1} a^1_{\mu \ga}(\z) z_{\mu} + \displaystyle \sum_{\nu=1}^{n-1} a^1_{\ga\nu}(\z) z_{\nu}\end{pmatrix}_{\ga=1, \ldots, n-1}& 1 \end{pmatrix}.
\end{equation}
Thus, as $\z \to \z_0$, we also have $(\phi_2^{\z})'(z) \to (\phi_2^{\z^0})'(z)$ in norm and uniformly on compact subset of $\mf{C}^n$.

\medskip

\item In the current coordinates, $\pa D$ is strongly pseudoconvex at $\zeta=0$ and the complex tangent space to $\pa D$ at $\z$ is $H_\zeta(\pa D)=\{z_n=0\}$. Therefore, the Hermitian form form $H^{\z}_2({'z},0)$ is strictly positive definite. Let $\la_1^{\z}, \ldots, \la^{\z}_{n-1}$ be its eigenvalues. Thus, we can choose a linear change of coordinates of the form $w=\phi^{\z}_3(z)=\La_{\z}U_{\z}z$, where $U_{\z}$ is a unitary rotation which keeps the last coordinate unchanged and $\La_{\z}=\diag(\sqrt{\la_1^{\z}}, \ldots, \sqrt{\la_{n-1}^{\z}}, 1)$, so that in these coordinates, again denoted by $z$ itself, the defining function $\rho^{\z}=\rho\circ (\phi_1^\z)^{-1}\circ (\phi_2^{\z})^{-1} \circ (\phi_3^{\z})^{-1}$ of the domain $\phi_3^{\z}\circ\phi_2^{\z}\circ \phi_1^{\z}(D)$, near the origin, has the form (ii) in the statement of the lemma. By definition, note that $\phi^{\z}_3$ fixes the points on the $\Re z_n$-axis. The linear map $\phi_3^{\z}$ can clearly be chosen so that it depends continuously on $\z$, and its derivative at any point being itself also depends continuously on $\z$.
 \end{itemize}
 
The map $h_{\z}$ is the composition $h_{\z}=\phi^{\z}_3 \circ \phi^{\z}_2 \circ \phi^{\z}_1$. It is evident from its construction that $h_{p^0}=I$, $h_{\z}(\z)=0$, and $\rho^{\z}=\rho\circ h_{\z}^{-1}$ has the desired form as in (iii). Moreover, since $\phi_1^{\z}$ maps $n_{\z}$ to the $\Re z_n$-axis, and both $\phi^{\z}_2$ and $\phi^{\z}_3$ fix points on the $\Re z_n$-axis, it follows that $h_{\z}$ maps $n_{\z}$ to the $\Re z_n$-axis which is the real normal to $\pa D_{\z}$ at the origin. More specifically, we have

\begin{lem}\label{h-h'-n_z}
With notations as in Lemma~\ref{pin}, and if $p \in n_{\z} \cap D$, then
\begin{equation}\label{h-n_z}
h_{\z}(p)= \big(0, -\vert p-\z\vert \vert \nabla_{\ov z} \rho (\z) \vert\big)
\end{equation}
and
\begin{equation}\label{h'-n_z}
h_{\z}'(p)=\La_{\z}\mf{U}_{\z}\mf{P}_{\z}.
\end{equation}
\end{lem}
\begin{proof}
Note that
\[
p=\z-\vert p-\z\vert \frac{\nabla_{\ov z} \rho (\z)}{\vert \nabla_{\ov z} \rho (\z) \vert}.
\]
Therefore, by \eqref{phi1-n_z}
\[
\phi_1^{\z}(p) =\big(0, -\vert p-\z\vert \vert \nabla_{\ov z} \rho (\z) \vert\big).
\]
The first part of the lemma follows from the fact that $\phi_2^{\z}$ and $\phi_3^{\z}$ fix the points on the $\Re z_n$-axis. For the second part we only need to observe that $(\phi_2^{\z})' (\phi_1^{\z}(p))=\mbb{I}$ from \eqref{der-phi2} and that $(\phi_3^{\z})'(z)=\La_{\z}\mf{U}_{\z}$ for any $z$.
\end{proof}

%%%%%%%%%
\subsection{Scaling}
We now proceed to prove Theorem~\ref{spscvx}. Let $D \subset \mf{C}^n$ be a $C^2$-smoothly bounded strongly pseudoconvex domain and let $p^0 \in \pa D$. Then there exist local holomorphic coordinates $z_1, \ldots, z_n$ near $p^0$ in which $p^0=0$, a neighbourhood $U$ of $p^0=0$ such that $U \cap D=\{z \in U : \rho(z)<0\}$ where $\rho(z)$ is a smooth function on $U$ of the form
\begin{equation}\label{initial-df}
\rho(z) = 2\Re z_n + \vert {'}z\vert^2 + o(\vert 'z\vert^2, \Im z_n),
\end{equation}
that satisfies $\nabla \rho(z) \neq 0$ for all $z \in U$, and a constant $0<r<1$ such that 
\begin{equation}\label{D-inside-ball}
U\cap D \subset \Om:=\Big\{z \in \mf{C}^n : 2 \Re z_n + r \vert {'z}\vert^2 <0\Big\}.
\end{equation}
Henceforth, we will be working in the above coordinates. 

In view of the localization result Theorem~\ref{localise-peak}, by shrinking $U$ if necessary, it is enough to establish the theorem for $U \cap D$. We relabel $U\cap D$ as $D$ for convenience. Let $ p^j $ be a sequence of points in $D$ converging to $p^0=0$. Without loss of generality, we may assume that $p^j \in U$ for all $j$. Choose $\zeta^j \in \pa D$ closest to $p^j$, which is unique if $j$ is sufficiently large. Note that $\zeta^j \rightarrow p^0 $. Let $\de_j=\de_D(p^j)=\vert p^j-\z^j\vert$. Set $\phi_i^j=\phi^{\z_j}_i$, $h_j=h_{\zeta^j}$, $D_j=D_{\zeta_j}$, and $\rho_j=\rho^{\zeta^j}$. Then by Lemma~\ref{pin}, near $0$,
\begin{equation}\label{rho-j}
\rho_j(z)=2 \Re\big(z_n+Q_j(z)\big)+H_j(z)+R_j(z),
%\rho_j(z)=2 \Re \Big(z_n + \sum_{\mu, \nu=1}^n a_{\mu \nu} (\zeta^j) z_{\mu} z_{\nu} \Big) + \sum_{\mu, \nu=1}^n b_{\mu \ov \nu} (\zeta^j) z_{\mu} \overline{z}_{\nu} + R_{j}(z)
\end{equation}
where $Q_j=Q_{\z^j}$, $H_j=H_{\z^j}$, and $R_j=R_{\z^j}$. Moreover, thanks to the strong pseudoconvexity of $\pa D$ near $p^0=0$, shrinking $U$ and taking a smaller $r$ in  \eqref{D-inside-ball} if necessary, we have
\begin{equation}\label{D_j-inside-ball}
D_j \subset \Om
\end{equation}
for all $j$ large.

\medskip

Let $q^j=h^j(p^j)$ and $ \eta_j =\de_{D^j}(q^j)$. From Lemma~\ref{h-h'-n_z}, we have
\[
q^j=\big({'0},-\de_j \vert \nabla_{\ov z} \rho(\z^j)\vert \big),
\]
and thus $\eta_j=\de_j \vert \nabla_{\ov z} \rho(\z^j)\vert$. Therefore,
\begin{equation}\label{delta-eta}
\frac{\eta_j}{ \de_j } = \vert \nabla_{\ov z} \rho(\z^j)\vert \to \vert \nabla_{\ov z} \rho(p^0)\vert =1
\end{equation}
from \eqref{initial-df}. Also, let $\mf{S}_j=h_j'(p^j)$. % and $S_0=h_0'(p^0)$.
Note that Lemma~\ref{pin} gives
\begin{equation}\label{der-h^j}
\mf{S}_j\to \mf{I}
\end{equation}
in the operator norm.

Now consider the anisotropic dilation map $ T_j : \mf{C}^n \rightarrow \mf{C}^n $ 
defined by 
\[
 T_j( 'z, z_n) = \left(   \frac{'z}{\sqrt{\eta_j}}, \frac{z_n}{{\eta}_j } \right).
\]
Set
\begin{align*}
\ti D_j & = T_j(D_j),\\
\ti \rho_j(z) &= \frac{1}{\eta_j} \rho_j(\sqrt{\eta_j}\,{'z}, \eta_jz_n).
\end{align*}
Note that any compact subset of $\mf{C}^n$ is contained in the domain of $\ti \rho_j$ for all sufficiently large $j$ and $\ti \rho_j$ is a defining function for $\ti D_j$. By \eqref{rho-j}
\[
\ti \rho_j(z) = 2\Re z_n + \vert 'z\vert^2 + \frac{1}{\eta_j}A_j(\sqrt{\eta_j}{'z},\eta_j z_n),
\]
where
\[
A_j(z)=z_nO(\vert z\vert)+O(\vert z \vert^3).
\]
It follows that $\ti \rho_j$ converges in $C^2$-toplogy on compact subsets of $\mf{C}^n$ to
\[
\rho_{\infty}(z)=2\Re z_n + \vert 'z\vert^2.
\]
This implies that the domains $\ti D_j$ converge in the local Hausdorff sense to the domain  $D_{\infty}$ defined in the statement of Theorem~\ref{spscvx}.  Also note that since $T_j\circ h_j(p^j)=({'0},-1)=b^{*}$, each $\ti D_j$ contains the point $b^{*}$. We now derive a stability result for the Narasimhan-Simha weighted kernel under scaling. First we note the following properties of the Cayley transform $\Phi$ defined in \eqref{cayley} that follow by a routine calculation: $\Phi$ is a biholomorphism of the domain (of $\Phi$)
\[
\mathcal{D}_\Phi:=\mf{C}^n \setminus \{z: z_n=1\}
\]
onto itself with $\Phi^{-1}=\Phi$. The domain $\Om$ defined in \eqref{D-inside-ball} is mapped by $\Phi$ onto the bounded domain
\begin{equation}\label{Phi-Om}
\Phi(\Om)=\Big\{z \in \mf{C}^n: r\vert {'z}\vert^2 +\vert z_n\vert^2 <1\Big\},
\end{equation}
and the domain $D_{\infty}$ is mapped to
\begin{equation}\label{Phi-D-infty}
\Phi(D_{\infty})=\mf{B}^n.
\end{equation}

\begin{lem}\label{conv-der-ker}
For any multiindex $A=(a_1, \ldots, a_n)$,
\[
\pa^{A} K_{\ti D^j,d}(z) \to \pa^{A} K_{D_{\infty},d} (z)
\]
uniformly on compact subsets of $D_{\infty}$.
\end{lem}
\begin{proof}
By \eqref{D_j-inside-ball}, the domain $\Om$ contains $D_j$ for all sufficiently large $j$. Also, $\Om$ is invariant under the maps $T_j$. Therefore, $\Om$ contains $\ti D_j$ for all sufficiently large $j$. Evidently, $\Om$ contains $D_{\infty}$ (as $r<1$). We claim that $\Phi(\ti D^j)$ converge to $\Phi(D_{\infty})=\mf{B}^n$ in the way as required by the hypothesis of Theorem~\ref{RT} with $q=0$. Indeed, first note that if $K$ is a compact subset of $\Phi(D_{\infty})$, then $\Phi^{-1}(K)$ is a compact subset of $D_{\infty}$. Since $\ti D_j \to D_{\infty}$ in the local Hausdorff sense,  $\Phi^{-1}(K)$ is contained in $\ti D_j$ for all sufficiently large $j$ implying that $K$ is contained in $\Phi(\ti D_j)$ for all sufficiently large $j$. Next, assume, if possible, that the second condition in the hypothesis of Theorem~\ref{RT} is not satisfied with $q=0$. Then there exist an $\ep>0$, a subsequence of $\Phi(\ti D_j)$, which we relabel as $\Phi(\ti D_j)$, and $\xi^j \in \Phi(\ti D_j)$ such that $\xi^j$ lies outside $(1+\ep)\Phi(D_{\infty})$, i.e., $\vert \xi^j\vert \geq 1+ \ep$. Since $\Phi(\ti D_j) \subset \Phi(\Om)$ for all large $j$, the sequence $\{\xi^j\}$ is bounded and hence after passing to a subsequence, $\xi^j \to \xi$ for some $\xi \in \ov{\Phi(\Om)}$. This, first of all, implies that $\vert \xi_n\vert^2 +r \vert {'}\xi\vert^2 \leq 1$ by \eqref{Phi-Om} as well as $\vert \xi \vert \geq 1+\ep$, which together ensures that $\xi_n \neq 1$, i.e., $\xi \in \mathcal{D}_\Phi$, and second, it now also implies that $\Phi^{-1}(\xi_j) \to \Phi^{-1}(\xi)$. Since $\Phi^{-1}(\xi_j) \in \ti D_j$, we have $\ti \rho_j (\Phi^{-1}(\xi_j))<0$ for all large $j$, and hence $\rho_{\infty}(\Phi^{-1}(\xi)) \leq 0$. Therefore, $\Phi^{-1}(\xi) \in \ov D_{\infty}$ and hence $\xi \in \ov {\Phi(D_{\infty})}$, which contradicts the fact that $\vert \xi\vert \geq 1+\ep$. This proves our claim. Therefore, by Theorem~\ref{RT}, $K_{\Phi(\ti D_j),d}(z)$ converges to $K_{\Phi(D_{\infty}),d}(z)$ uniformly on compact subsets of $\Phi(D_{\infty})$, together with all derivatives. By the transformation rule, the lemma follows immediately.
\end{proof}

\subsection{Boundary asymptotics}
We are now ready to complete the proof of Theorem~\ref{spscvx}. Recall that $T_j \circ h_j : D \to \ti D_j$ sends $p^j$ to $b^{*}=('0,-1)$. The matrix of the linear map $T_j$ is $\mf{T}_j=\diag(1/\sqrt{\eta_j}, \ldots, 1/\sqrt{\eta_j}, 1/\eta_j)$ and
\begin{equation}\label{det-T_j}
\det \mf{T}_j=\eta_j^{-(n+1)/2}.
\end{equation}
Also recall that $\mf{S}_j=h_j'(p^j)$.  Thus $(T_j\circ h_j)^{\prime}(p^j)=\mf{T}_j\mf{S}_j$. In the rest of the proof below, we will be repeatedly using Corollary~\ref{NS-ball-comp-z}, \eqref{delta-eta}, \eqref{der-h^j}, Lemma~\ref{conv-der-ker}, and the Cayley transform $\Phi$ defined in \eqref{cayley} that maps $D_{\infty}$ biholomorphically onto $\mf{B}^n$ with $\Phi(b^*)=0$,
\begin{equation*}%\label{Phi-der-det}
\Phi'(b^*)=-\diag (1/\sqrt{2}, \ldots, 1/\sqrt{2}, 1/2), \quad \text{and} \quad
 \det \Phi'(b^{*})=(-1)^n2^{-(n+1)/2},
\end{equation*}
without referring to them.

\medskip

(a) We have
\begin{align*}
K_{D,d}(p^j) = K_{\ti D_j, d}(b^*) \vert \det \mf{T}_j\mf{S}_j\vert^{2d+d} = \eta_j^{-(d+1)(n+1)} K_{\ti D_j, d}(b^{*}) \vert\det \mf{S}_j\vert^{2d+2}.
\end{align*}
Therefore,
\begin{multline*}
\de_j^{(d+1)(n+1)} K_{D, d}(p^j)
= \left(\frac{\de_j}{\eta_j}\right)^{(d+1)(n+1)} K_{\ti D^j, d}(b^{*})  \vert\det S_j\vert^{2d+2} \\
\to K_{D_{\infty},d}\big(b^{*}\big) =  K_{\mf{B}^n,d}(0) \big\vert \det \Phi'(b^*)\big\vert^{2d+2} =c\left(\frac{n!}{2^{n+1}\pi^n}\right)^{d+1}.
\end{multline*}

(b) By \eqref{tr-2}, we have
\[
g_{D,d}(p^j) = g_{\ti D^j, d} \vert\det \mf{T}_j\mf{S}_j\vert^2 =\eta_j^{-(n+1)}g_{\ti D^j,d}(b^*)  \vert \det \mf{S}_j\vert^{2},
\]
and therefore,
\begin{multline*}
\de_j^{n+1} g_{D,d}(p^j) = \left(\frac{\de_j}{\eta_j}\right)^{n+1} g_{\ti D^j,d}(b^*)  \vert \det \mf{S}_j\vert^{2} \to  g_{D_{\infty},d}(b^*)\\ = g_{\mf{B}^n,d}(0) \big\vert \det \Phi'(b^*)\big\vert^2 =\frac{1}{2^{n+1}} (d+1)^n (n+1)^n.
\end{multline*}

(c) From (a) and (b),
\begin{multline*}
\be_{D,d}(p^j) = \be_{\ti D_j,d}(b^*) = \frac{g_{\ti D_j,d}(b^*)}{K_{\ti D_j,d}(b^*)^{1/(d+1)}} \to \frac{g_{D_{\infty},d}(b^*)}{K_{D_{\infty},d}\big(b^{*}\big)^{1/(d+1)}}\\
 = \be_{D_{\infty},d}(b^*) =\be_{\mbb{B}^n,d}(0) =(d+1)^{n}(n+1)^{n} \left(\frac{1}{c}\right)^{1/(d+1)} \frac{\pi^n}{n!}.
\end{multline*}

(d) By invariance of the metric
\begin{equation}\label{tau-v}
\tau_{D,d}(p^j,v)= \tau_{\ti D_j,d}\big(b^{*},\mf{T}_j \mf{S}_j v)
= \tau_{\ti D_j,d}\left(b^{*},\left(\frac{'(\mf{S}_jv)}{\sqrt{\eta_j}}, \frac{(\mf{S}_jv)_n}{\eta_j}\right) \right).
\end{equation}
Therefore,
\begin{multline*}
\de_j\tau_{D,d}(p^j,v) =\frac{\de_j}{\eta_j} \tau_{\ti D_j,d}\Big(b^{*},\big(\sqrt{\eta_j}\,{'(\mf{S}_jv)}, (\mf{S}_jv)_n\big) \Big)  \to \tau_{D_{\infty},d}\big(b^{*}, ('0, v_n)\big)\\
= \tau_{\mbb{B}^n,d}\big(0, \Phi'(b^*) ('0,v_n) \big)=\tau_{\mbb{B}^n,d} \big(0, ('0,-v_n/2)\big)=\frac{1}{2}\sqrt{(d+1)(n+1)}\big\vert v_N(p^0)\big\vert,
\end{multline*}
as $v_N(p^0)=({'}0,v_n)$ by \eqref{initial-df}.

(e) For brevity, write $v_H^j=v_H(p^j)$ and $v_H^0=v_H(p^0)=({'v},0)$ by \eqref{initial-df}. We claim that $\big(\mf{S}_j v_H^j\big)_n=0$. Indeed, note that from Lemma~\ref{h-h'-n_z},
\[
\mf{S}_jv_H^j = {\La}_{\z^j}\mf{U}_{\z^j}\mf{P}_{\z^j} v_H^j.
\]
From \eqref{phi1-n_z},
\begin{align*}
\mf{P}_{\z^j} v_H^j & = \mf{P}_{\z^j}\left(v-\left\langle v, \frac{\nabla_{\ov z} \rho(\z^j)}{\vert \nabla_{\ov z} \rho(\z^j)\vert}\right\rangle \frac{\nabla_{\ov z} \rho(\z^j)}{\vert \nabla_{\ov z} \rho(\z^j)\vert}\right)\\
&=\mf{P}_{\z^j}v-\big\langle v, \nabla_{\ov z} \rho(\z^j)\big\rangle \frac{1}{\vert \nabla_{\ov z} \rho(\z^j)\vert^2}\mf{P}_{\z^j}\nabla_{\ov z} \rho(\z^j)\\
& = \Big({'\mf{P}_{\z^j}v}, \big\langle v, \nabla_{\ov z} \rho(\z^j)\big\rangle\Big) - \Big(0, \big\langle v, \nabla_{\ov z} \rho(\z^j)\big\rangle\Big) =({'\mf{P}_{\z^j}v}, 0).
\end{align*}
Also, $\La_{\z^j}$ and $\mf{U}_{\z^j}$ do not effect the $z_n$-coordinate. Hence, it follows that $\mf{S}_jv_H^j=({'W^j},0)$ for some ${'W^j} \in \mbb{C}^{n-1}$ and the claim follows. Now from \eqref{tau-v},
\[
\tau_{D,d}\big(p^j,v_H^j\big) =\tau_{\ti D_j,d}\left(b^{*},\left(\frac{'(\mf{S}_jv_H^j)}{\sqrt{\eta_j}}, 0\right) \right),
\]
and hence
\begin{align*}
\sqrt{\de_j}\tau_{D,d}\big(p^j,v_H^j\big)
&=\sqrt{\frac{\de_j}{\eta_j}} \tau_{\ti D_j,d}\Big(b^*,\big({'(\mf{S}_jv_H^j)},0\big) \Big) \to \tau_{D_{\infty},d}\big(b^{*}, ({'}v,0)\big)\\
&= \tau_{\mbb{B}^n,d}\big(0, \Phi'(b^*)({'}v,0)\big) =  \tau_{\mbb{B}^n,d}\big(0, ({-'}v/\sqrt{2},0)\big)\\
& =\frac{1}{\sqrt{2}}\sqrt{(d+1)(n+1)} \big\vert {'}v\big\vert\\
& =\sqrt{\frac{1}{2}(d+1)(n+1)\mathcal{L}_\rho(p^0, v_H^0)},
\end{align*}
as $\vert {'}v\vert^2 = \vert {'}v_H^0\vert^2=\mathcal{L}_\rho(p^0, v_H^0)$ by \eqref{initial-df}.

\medskip

(f) Let $\hat{v}= \lim_{j \to \infty} (\mf{T}_j\mf{S}_jv)/\vert \mf{T}_j\mf{S}_jv\vert$. Then
\begin{multline*}
R_{D,d}(p^j,v) =R_{\ti D_j,d}(b^*, \mf{T}_j\mf{S}_jv) = R_{\ti D_j,d}\left(b^*, \frac{\mf{T}_j\mf{S}_jv}{\vert \mf{T}_j\mf{S}_jv\vert}\right) \\ \to R_{D_\infty,d}\big(b^*, \hat{v}\big)
=R_{\mf{B}^n,d}\big(0, \Phi'(b^*)\hat{v}\big) =-\frac{2}{(d+1)(n+1)}.
\end{multline*}

(g) Proceeding as in (f), we obtain
\[
\Ric_{D,d}(p^j,v) \to \Ric_{\mf{B}^n,d}\big(0, \Phi'(b^*)\hat{v}\big) = -\frac{1}{d+1}.
\]
This completes the proof of the Theorem~\ref{spscvx}.
%%%%%%%%%%
\section{Proof of Theorem 1.2}

\noindent First, note that for a bounded domain $D \subset \mf C^n$, it is known that (see for example \cite{JP}) the Bergman metric always dominates the Carath\'{e}odory metric. By using the relationship between $K_{D, d}, \tau_{D, d}$ with the minimum integrals as discussed in Section 3.1, the same proof applies to show that $ds^2_{D, d}$ also dominates the Carath\'{e}odory metric on $D$. In particular, if $D$ is strongly pseudoconvex, it is known that the Carath\'{e}odory metric is complete, and hence the Narasimhan--Simha metric $ds^2_{D, d}$ must also be complete for every $d \ge 0$.

%%%%%%%%%
\section{Remarks and questions}

\noindent (i) It is well known that the paradigm of the scaling principle applies to more general domains. It is therefore possible to formulate analogs of Theorem~\ref{spscvx} for such domains, but to do so, it is essential to understand the weighted kernels on unbounded model domains. As an example, consider the unbounded domain
\[
D_{\infty} = \big\{ (z_1, z_2) \in \mf C^2: 2 \Re z_2 + P(z_1, \ov z_1) < 0 \big\}
\]
where $P = P(z_1, \ov z_1)$ is a subharmonic polynomial without harmonic terms. Such a domain arises as the limiting model domain associated with a smooth weakly pseudoconvex finite type domain in $\mf C^2$. What can be said about $K_{D_{\infty}, d}$ when $d \geq1$? In particular, is it true that $K_{D_{\infty}, d}$ has a positive lower bound everywhere? Knowing this would be helpful in controlling the weighted kernel of the scaled domains and would lead to conclusions similar to those listed in Theorem 1.1 for smooth weakly pseudoconvex finite type domains. Are there analogs of Catlin's results (\cite{C}) for $K_{D, d}$?

\medskip

\noindent (ii) It is natural to formulate a Lu Qi-Keng type conjecture for $K_{D, d}(z, w)$, $d \geq 1$. Does that also fail generically? -- see \cite{Boas-zeros} for the case $d=0$.

\medskip

\noindent (iii) Are there pseudoconvex domains for which the weighted Bergman spaces $A^2_d(D)$, $d \geq 1$, are non trivial and finite dimensional? -- see \cite{Wie} for the case $d=0$.

%%%%%%%%%%%%%%%%
%%%%%%%%%%%%%%%%

\end{document}